\documentclass[11pt]{amsart}

\usepackage{amscd,latexsym,amsthm,amsfonts,amssymb,amsmath,amsxtra,dsfont,braket,bbm,bm}

\usepackage{ytableau}
\ytableausetup{smalltableaux}

\usepackage{dynkin-diagrams}

\usepackage{multirow}
\renewcommand{\arraystretch}{1.5}

\usepackage[all]{xy}

\usepackage[mathscr]{eucal}
\usepackage[pdftex,bookmarks,colorlinks,pdfmenubar]{hyperref}

\usepackage{verbatim}

\renewcommand\theequation{\thesection.\arabic{equation}}

\usepackage{tikz-cd}
\usepackage{ytableau}
\ytableausetup{centertableaux}

\def\inn#1#2{\left\langle #1, #2\right\rangle}

\DeclareMathOperator{\AV}{AV}

\newcommand{\BC}{{\mathbb {C}}}

\newcommand{\BG}{{\mathbb {G}}}

\newcommand{\BR}{{\mathbb {R}}}

\newcommand{\BZ}{{\mathbb {Z}}}

\renewcommand{\CD}{{\mathcal {D}}}

\newcommand{\CG}{{\mathcal {G}}}

\newcommand{\CJ}{{\mathcal {J}}}
\newcommand{\CK}{{\mathcal {K}}}

\newcommand{\CN}{{\mathcal {N}}}
\newcommand{\CO}{{\mathcal {O}}}

\newcommand{\CS}{{\mathcal {S}}}

\newcommand{\CV}{{\mathcal {V}}}

\newcommand{\CX}{{\mathcal {X}}}
\newcommand{\CY}{{\mathcal {Y}}}
\newcommand{\CZ}{{\mathcal {Z}}}

\newcommand{\Fc}{{\mathfrak {c}}}

\newcommand{\Fg}{{\mathfrak {g}}}
\newcommand{\Fh}{{\mathfrak {h}}}

\newcommand{\Fl}{{\mathfrak {l}}}
\newcommand{\Fm}{{\mathfrak {m}}}
\newcommand{\Fn}{{\mathfrak {n}}}

\newcommand{\Fq}{{\mathfrak {q}}}

\newcommand{\Fs}{{\mathfrak {s}}}

\newcommand{\Fu}{{\mathfrak {u}}}

\newcommand{\Fx}{{\mathfrak {x}}}
\newcommand{\Fy}{{\mathfrak {y}}}

\newcommand{\RU}{{\mathrm {U}}}

\newcommand{\Ad}{{\mathrm{Ad}}}

\newcommand{\ari}{{\mathrm{ari}}}

\newcommand{\av}{{\mathrm{AV}}}

\newcommand{\disc}{{\mathrm{disc}}}

\newcommand{\des}{{\mathrm{des}}}

\newcommand{\Gal}{{\mathrm{Gal}}}
\newcommand{\GL}{{\mathrm{GL}}}

\newcommand{\Hom}{{\mathrm{Hom}}}

\newcommand{\ind}{{\mathrm{ind}}}

\newcommand{\Isom}{{\mathrm{Isom}}}

\newcommand{\Mp}{{\mathrm{Mp}}}

\newcommand{\mx}{\mathrm{max}}

\newcommand{\pr}{{\mathrm{pr}}}

\newcommand{\Res}{{\mathrm{Res}}}

\newcommand{\SO}{{\mathrm{SO}}}

\newcommand{\SU}{{\mathrm{SU}}}

\newcommand{\sgn}{{\mathrm{sgn}}}
\newcommand{\Sp}{{\mathrm{Sp}}}

\newcommand{\st}{{\mathrm{st}}}

\newcommand{\tr}{{\mathrm{tr}}}

\newcommand{\wf}{\mathrm{WF}}
\newcommand{\Wh}{\mathrm{Wh}}
\newcommand{\wm}{\mathrm{wm}}

\def\fpp{\mathfrak{p}}
\def\fbb{\mathfrak{b}}
\def\fqq{\mathfrak{q}}
\def\fll{\mathfrak{l}}
\def\fuu{\mathfrak{u}}
\def\fgg{\mathfrak{g}}
\def\pr{\mathrm{pr}}
\def\Gal{\mathrm{Gal}}

\newcommand{\udl}{\underline}

\newcommand{\wh}{\widehat}

\newcommand{\ol}{\overline}
\newcommand{\ul}{\underline}

\def\diag{{\rm diag}}

\def\sig{{\sigma}}

\newtheorem{thm}{Theorem}[section]
\newtheorem{cor}[thm]{Corollary}
\newtheorem{lem}[thm]{Lemma}
\newtheorem{prop}[thm]{Proposition}
\newtheorem {conj}[thm]{Conjecture}
\newtheorem {ques/conj}[thm]{Question/Conjecture}

\newtheorem{defn}[thm]{Definition}
\newtheorem{rmk}[thm]{Remark}

\newcommand{\be}{\begin{equation}}
\newcommand{\ee}{\end{equation}}

\usepackage[numbers]{natbib}

\newcommand{\ad}{\mathrm{ad}}

\begin{document}

\renewcommand{\theequation}{\arabic{equation}}
\numberwithin{equation}{section}

\title[Arithmetic Wavefront Set and Harish-Chandra Character]{Arithmetic Wavefront Set and Microlocal Structure of Harish-Chandra Character}

\author{Dihua Jiang}
\address{School of Mathematics, University of Minnesota, Minneapolis, MN 55455, USA}
\email{jiang034@umn.edu}

\author{Dongwen Liu}
\address{School of Mathematical Sciences, Zhejiang University, Hangzhou 310058, Zhejiang, P. R. China}
\email{maliu@zju.edu.cn}

\author[Z. Luo]{Zhikang Luo}
\address{Department of Mathematics,
National University of Singapore,
Singapore 119076}
\email{luozhikang@u.nus.edu}

\author[J.-J. Ma]{Jia-Jun Ma}
\address{School of Mathematical Sciences, Xiamen University, Xiamen, Fujian, P. R. China}
\email{hoxide@xmu.edu.cn}

\author{Lei Zhang}
\address{Department of Mathematics,
National University of Singapore,
Singapore 119076}
\email{matzhlei@nus.edu.sg}


\date{\today}

\thanks{ 
The work of D.J. is partially supported by the Simons Grant: SFI-MPS-TSM-00013449.
The work of D.L. is partially supported by National Natural Science Foundation of China No. 12171421. 
The work of L.Z. is supported by AcRF Tier 1 grant A-8002960-00-00 and Tier 2 grant T2EP20225-0047 of National University of Singapore.
}

\begin{abstract} 
In this paper, we establish in Theorem \ref{thm:RWS} the reciprocity of wavefront sets for irreducible admissible representations $\pi$ of classical groups $G$ over any local field $F$ of characteristic zero if $\pi$ has a generic local $L$-parameter. Over archimedean local fields, based on the progress made in \cite{JLZ25}, we prove in Theorem \ref{thm:wfsc} that for an irreducible Casselman-Wallach representation $\pi$ with a generic local $L$-parameter, the Wavefront Set Conjecture \cite[Conjecture 1.2]{JLZ25} and its refinement (Conjecture \ref{conj:wfss}) hold for the arithmetic wavefront set $\wf_\ari(\pi)$ as defined by the associated enhanced local $L$-parameter of $\pi$ and the wavefront set $\wf_\tr(\pi)$ defined by the Harish-Chandra distribution character $\Theta_\pi$ of $\pi$. Hence the microlocal structure of $\Theta_\pi$ is completely determined by the arithmetic information carried by the enhanced local $L$-parameter of $\pi$. The relations with the algebraic wavefront set 
$\wf_\wm(\pi)$ defined by the degenerate Whittaker models are extensively discussed by means of the composition law 
(Theorem \ref{Thm:CL}) over all local fields of characteristic zero. Under Conjecture \ref{conj:DDW}, the Wavefront
Set Conjecture is fully established over archimedean local fields. As a consequence, we prove a refinement of Vogan's maximal-orbit principle (Theorem \ref{thm:VC}). 
\end{abstract}

\maketitle

\tableofcontents

\section{Introduction}\label{sec:I}

As a sequel of our previous work (\cite{JLZ25}),
we prove in this paper a large and essential part of the wavefront set conjecture (\cite[Conjecture 1.2]{JLZ25}) for real classical groups. This brings substantial progress for a large family of irreducible admissible representations of real classical groups in the investigation of long-standing open problems in the representation theory of real reductive Lie groups on relations among the wavefront set of an irreducible admissible representation and other relevant invariants. The approach we took is based on the theory of local descents for irreducible admissible representations, which has been well developed for representations with generic local Langlands parameters, and puts for the first time the study of the long-standing open problems on wavefront sets and related problems within the framework of the local Langlands program. 

Around 1970, L. H\"ormander introduced an influential notion of wavefront sets for distributions over smooth manifolds in the theory of Fourier integral operators (\cite{Hor71, Hor90}). For any unitary representation $\pi$ of a real Lie group $G$, R. Howe introduced around 1980 the notion of the wavefront set of $\pi$  (\cite{H81}). He first proved that the wavefront set $\wf(\pi)$ of the Harish-Chandra distribution character $\Theta_\pi$ (\cite{HC65a, HC65b}) is $G$-invariant under the adjoint action and is determined by its image under the projection $G\times\Fg^*\to\Fg^*$, where 
$\Fg$ is the Lie algebra of $G$ and $\Fg^*$ is the dual space of $\Fg$. Hence the wavefront set of $\pi$ as defined by R. Howe is a subset of the nilpotent cone of $\Fg$ if $G$ is reductive and $\Fg^*$ is identified with $\Fg$ via the Killing form. 

As expected in \cite{H81, BV80}, for a real reductive group $G$, the microlocal structure of the wavefront set of an irreducible admissible representation $\pi$ of $G$ should be detected by the germ expression near the identity of the distribution character $\Theta_\pi$. Meanwhile, 
D. Vogan introduced the notion of associated variety $\av(\pi)$ of a Harish-Chandra module of a reductive Lie group and conjectured that $\av(\pi)=\wf(\pi)$, which provides an algebraic approach to understand the microlocal structure of the distribution character $\Theta_\pi$. It remains a difficult problem in theory to explicitly determine the structure of either the associated varieties $\av(\pi)$ or the wavefront sets $\wf(\pi)$. 

The expectation that an irreducible representation should have a well-defined maximal nilpotent support goes back to the Kostant-Kirillov orbit method philosophy and to the early work on primitive ideals.  Duflo's work on primitive ideals \cite{Duf77} provided one of the algebraic foundations for this viewpoint.  For complex semisimple Lie groups, the relationship between primitive ideals, orbital integrals, local character expansions, and nilpotent orbits was developed in the work of Barbasch--Vogan and Joseph; see, for example, \cite{BV80,BV82,BV83,BV85,Jos85}.  In this setting the maximal nilpotent support is governed by the closure of a single nilpotent orbit.  We shall refer to this principle as Vogan's maximal-orbit principle.  This viewpoint is also consistent with the Kirillov--Kostant orbit method for the unitary dual of real reductive groups, as explained for instance in \cite{V87,V97,V00,AV21}.

In general, let $F$ be an arbitrary local field of characteristic zero and $G$ be the $F$-rational points of a reductive algebraic group $\mathbb{G}$ over $F$. Denote by $\Pi_F(G)$ the set of equivalence classes of irreducible admissible representations of $G$, which in the archimedean case are of Casselman-Wallach type. 
When $F$ is non-archimedean,  for any $\pi \in \Pi_F(G)$, an asymptotic expansion of the distribution character $\Theta_\pi$ near the identity was established by R. Howe (\cite{H74}; see also \cite{He85, Prz90}) and Harish-Chandra (\cite{HC78}):
\begin{align}\label{thetapi}
\Theta_\pi=\sum_{\CO\in\CN_F(\Fg)_\circ}c_\CO\wh{\mu}_\CO,\quad c_\CO\in\BC,
\end{align}
where $\CN_F(\Fg)_\circ$ denotes the set of $F$-rational nilpotent orbits in the Lie algebra $\Fg$ of the group $G$ under the adjoint action of $G$ and
$\wh{\mu}_\CO$ denotes the Fourier transform of the (normalized) measure $\mu_\CO$ on the nilpotent orbit $\CO$. From \eqref{thetapi}, one defines the {\bf analytic wavefront set} of $\pi\in\Pi_F(G)$ to be
\begin{align}\label{wf-tr}
\wf_\tr(\pi):=\set{\CO\in\CN_F(\Fg)_\circ  |  c_\CO\neq 0},
\end{align}
where $c_\CO$'s are as in \eqref{thetapi}. When $F$ is an archimedean local field, following the work of D. Barbasch and D. Vogan (\cite{BV80}), one may define $\wf_\tr(\pi)$ in a similar way. 

Over a non-archimedean local field $F$, C. M\oe glin and J.-L. Waldspurger define in \cite{MW87} (see also \cite{K87}) the 
{\bf algebraic wavefront set} $\wf_\wm(\pi)$ of $\pi\in\Pi_F(G)$ by means of the {\it generalized Whittaker models} of $\pi$ associated with any nilpotent element $f$ in the set $\CN_F(\Fg)$ of all nilpotent elements in $\Fg$ (see Section \ref{sec:GW} for some details).
One of the main results in \cite{MW87,Vs14} shows that under the $F$-rational topological order (Definition \ref{defn:order}), the set of the maximal members, which is denoted by $\wf_\tr(\pi)^{r-\mx}$, of the analytic wavefront set $\wf_\tr(\pi)$ coincides with the set of maximal members, which is denoted by $\wf_\wm(\pi)^{r-\mx}$, of the algebraic wavefront set $\wf_\wm(\pi)$, that is,
\begin{align}\label{alg-ana-wfs}
\wf_\tr(\pi)^{r-\mx}=\wf_\wm(\pi)^{r-\mx}
\end{align}
holds for any $\pi\in \Pi_F(G)$. It is important to mention that over any non-archimedean local field $F$, a wavefront set may not be the closure of a single nilpotent orbit, due to the examples constructed by C.-C. Tsai (\cite{T24,T25}), although it remains the closure of a single nilpotent orbit in many cases. 

It is natural to extend the notion of the generalized Whittaker models to any irreducible Casselman-Wallach representations of real reductive groups and to conjecture that the assertion as in \eqref{alg-ana-wfs} holds for the archimedean case. However, only a few special situations in the archimedean case have been studied (see \cite{V78, Mat87, Mat90, Mat92, GS15, GGS17, AA26} for instance). 
It should be mentioned that when $F$ is archimedean, there are more invariants (associated variety, characteristic cycle and wavefront cycle, etc.) of $\pi$ closely related to the wavefront set of $\pi$ (\cite{SV00, V17}). These invariants form a rich theory. We refer to the papers of W. Rossmann (\cite{R95a,R95b}), of D. Vogan (\cite{V17}), of J. Adams and D. Vogan (\cite{AV21}), and of W. Schmid and K. Vilonen (\cite{SV00}) for excellent surveys and detailed discussions of the theory, in particular, for the relation between the explicit computation of the wavefront sets and the construction of the unitary dual of the real reductive groups via the Kirillov-Kostant orbit method (\cite{Kos70, V00, Kir04, MB25, LY23, DMB25}).  

It is important to point out that in \cite{AV21}, Adams and Vogan found, when $F$ is archimedean, 
a sophisticated and comprehensive algorithm,  based on certain profound geometric structures behind the wavefront sets, to compute the $\ol{F}$-structure of the wavefront set $\wf_\tr(\pi)$ for $\pi\in\Pi_F(G)$  via the complex-geometrical parametrization of $\Pi_F(G)$. It remains a difficult problem to determine the $F$-rational structure of the wavefront sets in general. 

For families of special unipotent representations of $G$, some substantial progress on determination of the $F$-rational structure of the wavefront sets was made by J.-L. Waldspurger (\cite{W18, W19, W20}) and by D. Ciubotaru, L. Mason-Brown, E. Okada  and others (\cite{CMBO21, CMBO24a, CMBO24b, CMBO25}) for the non-archimedean case, and by D. Barbasch, J.-J. Ma, B. Sun and C.-B. Zhu (\cite{BMSZ25, BMSZ}) for the archimedean case. Those results are essential to the explicit construction of the unipotent part of the unitary dual $\wh{G}$ of $G$, which is regarded as the building block of the whole unitary dual. For other families of representations, we refer to relevant papers (\cite{M96,Hb12,Q14,HHO16,HW17,Hb18,AGS24}), for instance. 

It is worthwhile to mention that even weaker information on the wavefront sets than the complete $F$-rational structure may find substantial applications. For instance, in the twisted automorphic descent theory of D. Jiang and L. Zhang (\cite{JZ20}), the information of wavefront sets played a substantial role in their construction of explicit automorphic modules for irreducible cuspidal automorphic representations of classical groups and in their proof of the global Gan-Gross-Prasad conjecture. It is also used in the work of 
D. Jiang and B. Liu (\cite{JL18,JL25}) to determine which global Arthur packets have no cuspidal members. The local upper bound conjecture on the wavefront sets by D. Jiang (\cite{J14}) and its variants have been extensively studied by B. Liu, C.-H. Lo and F. Shahidi (\cite{LLS24, LS22, LS26}) and by H. Atobe, D. Ciubotaru and J.-L. Kim (\cite{AC26, CK24, CK25}). Those local results are substantial towards the understanding of the local Arthur packets and the unitary dual problem for $G$. 

\subsection{Wavefront set and Langlands reciprocity}\label{ssec:WFS-LR}
In our previous paper (\cite{JLZ25}), we assume that $F$ is a local field of characteristic zero. 
In order to discover a practical way to determine the $F$-rational structure of the wavefront sets, based on our previous work (\cite{JZ18}), we developed the theory of local descents for enhanced local $L$-parameters and proposed in \cite{JLZ25} to utilize the local Langlands correspondence and the local Gan-Gross-Prasad conjecture as a bridge to transfer the problem of explicit computation of wavefront sets for a representation $\pi$ to that on the Langlands dual side and consider the structure of the arithmetic wavefront set associated to the enhanced local $L$-parameter of $\pi$, which seems quite productive and fruitful. 

In \cite{JLZ25}, for any classical group $\BG$ defined over $F$ and $G=\BG(F)$, we introduced the notion of {\bf arithmetic wavefront set} $\wf_\ari(\pi)$ 
(\cite[Theorem 4.4]{JLZ25}) for any $\pi\in\Pi_F(G)$ with generic local $L$-parameter (see Section \ref{sec:AW}) and made the following {\bf Wavefront Set Conjecture}.

\begin{conj}[Wavefront Set]\label{conj:wfss}
Let $F$ be a local field of characteristic zero and $G$ be a classical group over $F$. For any $\pi\in\Pi_F(G)$ with a generic local $L$-parameter, the following identities hold:
\begin{align}\label{s-max}
    \wf_\tr(\pi)^{s-\mx}=\wf_\ari(\pi)^{s-\mx}=\wf_\wm(\pi)^{s-\mx},
\end{align}
where $\wf_\square(\pi)^{s-\mx}$ denotes the subset of all maximal members in the wavefront set $\wf_\square(\pi)$ under the $F$-stable topological order over $\CN_F(\Fg)_\circ$
given in Definition \ref{defn:order}. Moreover, the following (stronger) identities also hold:
\begin{align}\label{r-max}
    \wf_\tr(\pi)^{r-\mx}=\wf_\ari(\pi)^{r-\mx}=\wf_\wm(\pi)^{r-\mx},
\end{align}
where $\wf_\square(\pi)^{r-\mx}$ denotes the subset of all maximal members in the wavefront set $\wf_\square(\pi)$ under the $F$-rational topological order over $\CN_F(\Fg)_\circ$.
\end{conj}

To illustrate the spirit of the local Langlands reciprocity in the theory of local descents for both representations $\pi$ of $G(F)$ and their associated enhanced local $L$-parameters we established (joint with C. Chen) in \cite[Theorem 1.4]{CJLZ} the coincidence of the first occurrence indices of the local descents for the representations and their enhanced local $L$-parameters when the local $L$-parameters are generic. In order to understand Conjecture \ref{conj:wfss}, we introduce in this paper a new notion of the wavefront set of $\pi$ when a $\pi\in\Pi_F(G)$ has a generic local $L$-parameter by the local descents. It is called the {\bf descent wavefront set} of $\pi$ and is denoted by $\wf_\des(\pi)$. We refer to Section \ref{sec:AW} for details. The spirit of the local Langlands reciprocity for the wavefront sets is illustrated in the following theorem. 

\begin{thm}[Reciprocity of Wavefront Set] \label{thm:RWS}
Let $F$ be a local field of characteristic zero and $G$ be an $F$-quasisplit classical group or its pure inner form. 
For any $\pi\in\Pi_F(G)$ with a generic enhanced $L$-parameter $(\varphi,\chi)$, the wavefront set $\wf_\des(\pi)$ 
defined by the consecutive descents of $\pi$ and the wavefront set $\wf_\ari(\pi)$ defined by the consecutive descents of $(\varphi,\chi)$ are equal, i.e. 
    \[
    \wf_\des(\pi)=\wf_\ari(\pi).
    \]
\end{thm}

This theorem will be restated as Theorem \ref{thm:des=ari} and proved in Section \ref{sec:AW}. By the composition law of descents as established in Theorem \ref{Thm:CL}, we obtain in Corollary \ref{cor:wf} that 
\begin{equation}\label{des-wm}
 \wf_\des(\pi)\subset\wf_\wm(\pi),\quad {\rm for\ any}\ \pi\in\Pi_F(G), 
\end{equation} 
with a generic $L$-parameter and over all local fields $F$ of characteristic zero. 
Conjecture \ref{conj:wfss} naturally leads to the following assertion.

\begin{conj}[Descent and Generalized Whittaker]\label{conj:DDW}
Let $F$ be a local field of characteristic zero and $G$ be a classical group over $F$. For any $\pi\in\Pi_F(G)$ with a generic local $L$-parameter, the following identities hold:
\begin{align*}
    \wf_\des(\pi)^{s-\mx}=\wf_\wm(\pi)^{s-\mx} \quad {\rm and}\quad \wf_\des(\pi)^{r-\mx}=\wf_\wm(\pi)^{r-\mx},
\end{align*}
where the notations take the same meaning as in Conjecture \ref{conj:wfss}.
\end{conj}

It is worthwhile to point out that Conjecture \ref{conj:DDW} is of representation-theoretic nature and is expected to be proved through an approach based on representation theory and combinatorics. It is even more mysterious if one considers the relations between $\wf_\des(\pi)$ and $\wf_\wm(\pi)$ for the boundary nilpotent orbits of the maximal ones. 

In Section \ref{ssec-MRR}, we explain the indispensable roles played by this reciprocity of wavefront sets (Theorem \ref{thm:RWS}) in our approach to the Wavefront Set Conjecture (Conjecture \ref{conj:wfss}) in the archimedean case.

\subsection{Main results over $\BR$}\label{ssec-MRR}
As indicated in our previous work (\cite{JLZ25}), the arithmetic wavefront set $\wf_\ari(\pi)$ for 
$\pi\in\Pi_F(G)$ with a generic local $L$-parameter is defined by means of the consecutive descents of the enhanced local Langlands datum $(\varphi,\chi)$ associated with $\pi$, based on the $F$-rational structure of the local Langlands correspondence and the local Gan-Gross-Prasad conjecture. The $F$-rational structure of $\wf_\ari(\pi)$ is expected to be completely determined and explicitly computed by the arithmetic information of $(\varphi,\chi)$. This expectation was completely realized in \cite[Theorem 9.2]{JLZ25} for all classical groups over the archimedean local fields. For later use, we record the following consequence of \cite[Theorem 9.2]{JLZ25} and \cite[Theorems 1.4, 7.14]{JLZ25}.

\begin{thm}[Arithmetic Wavefront Set over $\BR$]\label{FRS-WFS}
    Let $F$ be an archimedean local field and let $(\varphi,\chi)$ be an enhanced local $L$-parameter with $\varphi$ generic for a classical group $G$ over $F$. Assume that $\pi\in\Pi_F(G)$ is associated with $(\varphi,\chi)$. 
    \begin{enumerate}
        \item If the component group $\CS_\varphi$ associated with $\varphi$ is trivial, then $\wf_\ari(\pi)^{r-\mx}$ consists of all 
        $F$-rational regular nilpotent orbits in $\CN_F(\Fg)_\circ$.
        \item If the component group $\CS_\varphi$ associated with $\varphi$ is non-trivial, then $\wf_\ari(\pi)^{r-\mx}=\{\CO(\pi)\}$, where $\CO(\pi)$ is the unique $F$-rational nilpotent orbit that can be explicitly and effectively determined by $(\varphi,\chi)$ via the $L$-parameter descent method. 
        \item Every $F$-rational nilpotent orbit $\CO$ in $\wf_\ari(\pi)^{r-\mx}$ is $F$-distinguished in the sense that $\CO$ does not meet any $F$-rational proper Levi subalgebra of the Lie algebra $\Fg$ of $G$.
    \end{enumerate}
\end{thm}

It is important to mention that over a non-archimedean local field $F$, the $F$-rational structures of the arithmetic wavefront set $\wf_\ari(\pi)$ can be explicitly computed for various lower rank cases, however, it is more complicated as indicated in \cite{T24, T25, JLZ25} for instance. It is a subject matter of our ongoing project. We refer to \cite{JLZ25} for some relevant discussions. 

The objective of this paper is to prove a large part of Conjecture \ref{conj:wfss} (the Wavefront Set Conjecture) for all classical groups over archimedean local fields. When $F=\BC$, the field of complex numbers, Conjecture \ref{conj:wfss} holds trivially, and hence for the archimedean case considered in this paper we assume that $F=\BR$, the field of real numbers. Here is the main result of this paper (Theorem \ref{Thm:WFS}).

\begin{thm} 
\label{thm:wfsc}
Let $\BG$ be a classical group defined over $\BR$, and denote by $G$ the group $\BG(\BR)$ of $\BR$-rational points or its metaplectic double cover when $\BG$ is symplectic. 
Assume that $\pi\in\Pi_\BR(G)$ has a generic local $L$-parameter $\varphi$ under the local Langlands correspondence. 
If the component group $\CS_\varphi$ associated with $\varphi$ is trivial, then Conjecture \ref{conj:wfss} holds. If the component group $\CS_\varphi$ is non-trivial, then the following relations  
\[
\wf_\tr(\pi)^{r-\mx}=\wf_\ari(\pi)^{r-\mx} \subset \wf_\wm(\pi)^{s-\mx}
\]
hold. Moreover, the $\BR$-rational nilpotent orbits in $\wf_\wm(\pi)^{s-\mx}$ lie in a single stable nilpotent orbit. 
\end{thm}

With Theorem \ref{thm:wfsc}, the analytic wavefront set $\wf_\tr(\pi)^{r-\mx}$ can be computed explicitly and effectively, and produces members of $\wf_\tr(\pi)^{r-\mx}$ through the arithmetic wavefront set $\wf_\ari(\pi)^{r-\mx}$ by means of the enhanced local $L$-parameter $(\varphi,\chi)$ of $\pi$,  for a large family of $\pi\in\Pi_F(G)$, i.e. for $\pi$ having a generic local $L$-parameter. Hence the microlocal structure of the Harish-Chandra distribution character $\Theta_\pi$ can be completely determined by the arithmetic information associated with $\pi$ via the local Langlands reciprocity. We strongly believe that with a full development of the compatibility theory of the local descents and the local Langlands conjecture, this arithmetic approach to determine the wavefront sets $\wf_\wm(\pi)^{r-\mx}$ 
and $\wf_\tr(\pi)^{r-\mx}$ could be extended to the general situation. 

By the reciprocity of wavefront sets (Theorem \ref{thm:RWS}), we obtain the Wavefront Set Conjecture 
(Conjecture \ref{conj:wfss}) if we assume that Conjecture \ref{conj:DDW} holds. 

\begin{cor}\label{cor:WFC}
    Let $\BG$ be a classical group defined over $\BR$ and denote by $G=\BG(\BR)$ the $\BR$-rational points or its metaplectic double cover when $\BG$ is symplectic. 
Assume that $\pi\in\Pi_\BR(G)$ has a generic local $L$-parameter $\varphi$ under the local Langlands correspondence. Then Conjecture \ref{conj:DDW} implies Conjecture \ref{conj:wfss}. 
\end{cor}

Combining Theorems \ref{FRS-WFS} and \ref{thm:wfsc}, we obtain an arithmetic refinement of Vogan's maximal-orbit principle.  In the generic real classical case considered here, the maximal analytic wavefront support is determined by the arithmetic wavefront set attached to the enhanced local $L$-parameter.  In particular, in the nontrivial component-group case, the maximal analytic wavefront orbit is a single real rational nilpotent orbit, and this orbit is computed explicitly from the enhanced local $L$-parameter.


\begin{thm}[An arithmetic refinement of Vogan's maximal-orbit principle]\label{thm:VC}
Let $G$ be a classical group over $\BR$ as in Theorem \ref{thm:wfsc} and $(\varphi,\chi)$ be an enhanced local $L$-parameter for $G$ with $\varphi$ generic. Then a refinement of the Vogan conjecture holds for the representation $\pi\in\Pi_\BR(G)$ associated with $(\varphi,\chi)$. More precisely, if the component group 
$\CS_\varphi$ is trivial, then $\wf_\wm(\pi)^{r-\mx}$ and $\wf_\tr(\pi)^{r-\mx}$ consist of all $\BR$-rational regular nilpotent orbits in $\CN_F(\Fg)_\circ$; and if the component group 
$\CS_\varphi$ is non-trivial, then $\wf_\tr(\pi)^{r-\mx}$ consists of a unique 
$\BR$-rational $\BR$-distinguished nilpotent orbit $\CO(\pi)$ that is contained in $\wf_\wm(\pi)^{s-\mx}$ and can be explicitly and effectively determined by $(\varphi,\chi)$ via the $L$-parameter descent method. 
\end{thm}

Let us also point out that this maximal-orbit principle is a feature of the real-group setting and should not be interpreted as a literal statement over all local fields.  In the $p$-adic case, Tsai has shown that the geometric wavefront set need not be a singleton \cite{T24}.  This does not contradict the orbit method philosophy; rather, it indicates that the $p$-adic form of the philosophy is more genuinely arithmetic.  The leading asymptotic behavior of the Harish-Chandra distribution character $\Theta_\pi$ for some $\pi\in\Pi_F(G)$ is still governed by nilpotent orbital data, as in the local character expansion and in the relation with generalized Whittaker models, but the passage from rational nilpotent orbits to geometric nilpotent orbits need not collapse to a single maximal orbit. Thus the correct $p$-adic analogue of the orbit method should allow a finite set of maximal orbits, reflecting the $p$-adic nature of the rational structures associated with representations $\pi\in\Pi_F(G)$. 
From this viewpoint, Tsai's examples refine rather than refute the orbit method picture: they show that over non-archimedean fields the leading nilpotent support is orbital but not necessarily geometrically irreducible.

\quad

The idea of the proof of Theorem \ref{thm:wfsc} can be briefly outlined as follows. As mentioned earlier, the general comparison between $\wf_\wm(\pi)$ and $\wf_\tr(\pi)$ in the archimedean case is a long-standing open problem. The work of W. Schmid and K. Vilonen (\cite{SV00}) compares wavefront cycles with associated cycles, and thereby relates $\wf_\tr(\pi)$ to the associated variety
${\rm AV}(\pi)$ of $\pi$. The arithmetic wavefront set $\wf_\ari(\pi)$ introduced in \cite{JLZ25} builds up an important bridge connecting $\wf_\wm(\pi)$ and $\wf_\tr(\pi)$, which at the same time is also tied to ${\rm AV}(\pi)$. 
As we explained above, the arithmetic wavefront set $\wf_\ari(\pi)$ introduced in \cite{JLZ25} is defined by the consecutive $L$-parameter descents from the enhanced local $L$-parameter $(\varphi, \chi)$ of $\pi$ and hence is an arithmetic object. In Definition \ref{dfn:DWS}, we introduce the descent wavefront set $\wf_\des(\pi)$ of $\pi$ on the representation side and prove in Theorem \ref{thm:des=ari} the reciprocity of wavefront sets:  $\wf_\des(\pi)=\wf_\ari(\pi)$ holds for any $\pi\in\Pi_F(G)$ with a generic $L$-parameter. 
Then we prove a composition law (Theorem \ref{Thm:CL}) for generalized Whittaker models using a generalized root exchange lemma (Lemma \ref{lem:exchange}), which implies that (Corollary \ref{cor:wf})
\be \label{subset}
\wf_\ari(\pi)\subset \wf_\wm(\pi).
\ee
 We emphasize that the results in Section \ref{sec:WS} and Section \ref{sec:CL} hold for arbitrary local fields of characteristic zero,
and also have applications for the wavefront sets in the non-archimedean case. 


 By comparing the algorithm for $\wf_\ari(\pi)^{r-\mx}$ in terms of enhanced $L$-parameters given in \cite{JLZ25} with the algorithm for associated varieties of $A_\Fq(\lambda)$ modules in \cite{Tr05},  we show that (Theorem \ref{thm:AV})
\be \label{AV}
    {\rm AV}(\pi) = \bigcup_{\CO\in \wf_\ari(\pi)^{r-\mx}}\overline{\CO_{\mathfrak p}},
\ee
    where $\CO_{\mathfrak p}$ denotes the image of $\CO$ under the Kostant-Sekiguchi correspondence. 
Combined with the fundamental result of Schmid and Vilonen (\cite{SV00}), it follows that (Corollary \ref{AWF=AV})
    \[
    \wf_\tr(\pi)^{r-\mx}=\wf_\ari(\pi)^{r-\mx}.
    \]
    
    Let $\CV({\rm Ann}\,\pi)\subset \frak g_\BC$ be the associated variety of the annihilator ideal in $U(\frak g_\BC)$ of 
    $\pi$, where $\frak g_\BC\cong \Fg_\BC^*$ is the complexification of $\Fg$ with Cartan decomposition 
    $\frak g_\BC =\frak k \oplus \frak p$. A result of H. Matumoto (\cite[Corollary 4]{Mat87}) shows that the members of $\wf_\wm(\pi)$ all lie in 
    $\CV({\rm Ann}\,\pi)$.
By \cite[Corollary 4.7, Theorem 8.4]{V91}, $\CV({\rm Ann}\,\pi)$ is the Zariski closure of a single stable nilpotent orbit $\CO^{\rm st}(\pi)$
and 
\[
\av(\pi) = \overline{\CO^{\rm st}(\pi)\cap \av(\pi)} \subset \CV({\rm Ann}\,\pi)  \cap \frak p.
\]
From above discussions and \eqref{subset}, \eqref{AV}, it follows that
    \[
    \wf_\ari(\pi)^{r-\mx} \subset \wf_\wm(\pi)^{s-\mx},
    \]
    which completes the proof of Theorem \ref{thm:wfsc}.

In summary, the strategy of the proof can be roughly depicted as follows:
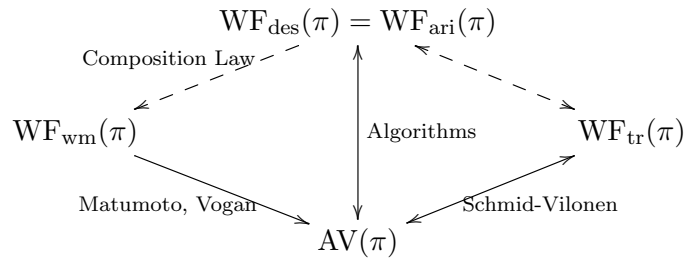
\[
\xymatrix{
& \wf_\des(\pi)=\wf_\ari(\pi) \ar@{-->}[dl]_{\text{Composition Law \ }}  \ar@{<->}[dd]^{\text{Algorithms}} \ar@{<-->}[dr] & \\
\wf_\wm(\pi) \ar[dr]_{\text{Matumoto, Vogan \ }} & & \wf_\tr(\pi) \\
& \mathrm{AV}(\pi) \ar@{<->}[ur]_{\text{ \ Schmid-Vilonen}} & 
}
\]


\quad

This paper is organized as follows. In Section \ref{sec:WS} we formulate the classification of nilpotent orbits and various versions of wavefront sets for classical groups.
In Section \ref{sec:CL} we present a generalized  root exchange lemma and use it to prove a composition law for generalized Whittaker models of classical groups over any local field of characteristic zero. We explicate the arithmetic wavefront sets for real classical groups based on our previous work \cite{JLZ25} in Section \ref{sec:AWF},
and explain the local Langlands correspondence and associated varieties for discrete series in Section \ref{sec:LLC}.
 The comparison between arithmetic wavefront sets and associated varieties is established in Section \ref{sec:AWFAV}. Combining these results, we prove the main theorem 
 in Section \ref{sec:PF}.

\section{Wavefront Sets of Classical Groups} \label{sec:WS}

In this section, we formulate the classification of nilpotent orbits and various versions of wavefront sets for classical groups, in order to make the paper self-contained. Let $F$ be a local field of characteristic zero as before. 
Let $F(\bm{\delta})$ be a quadratic extension of $F$, where $\bm{\delta}=\sqrt{d}$ for a non-square element $d\in F^\times$. Assume that 
$E=F$ or $F(\bm{\delta})$, and that $V$ is an $\epsilon$-Hermitian vector space of dimension $\Fn$ over $E$, where $\epsilon\in \{\pm 1\}$. 
Let  
\[
\BG(V):=\Isom(V)^\circ, \quad (\text{$\circ$ indicates the identity component)}
\]
and denote by $G(V)$ the $F$-rational points of $\BG(V)$. 

In this paper, we assume that $G$ equals either $G(V)$ or its metaplectic double cover $\Mp(V)$ when $E=F$ and $\epsilon=-1$.  The Lie algebra of 
$G$ is denoted by $\Fg =\Fg(V)$, and similar notation may be used without further explanation.  
Let $\CN_F(\Fg)$ be the set of all nilpotent elements in the Lie algebra $\Fg$ and $\CN_F(\Fg)_\circ$ be the set of $F$-rational nilpotent orbits under the adjoint action of $G$. 
For a nilpotent element $f\in \CN_F(\Fg)$, write $\CO_f\in \CN_F(\Fg)_\circ$ for the $F$-rational nilpotent orbit containing $f$. 

\subsection{Rational nilpotent orbits}\label{ssec:RNO}
It is a classical theorem that there are only finitely many nilpotent orbits of $\Fg$ over the algebraic closure $\ol{F}$ of $F$ by certain partitions determined by $\Fg$ (\cite{CM}). For the $F$-rational nilpotent orbits in $\Fg$, we follow \cite{GZ} to parametrize the set $\CN_F(\Fg)_\circ$ by sesquilinear Young tableaux. For $k\geq 1$, denote by $S_k$ the $k$-dimensional irreducible $\frak{sl}_2$-module over $E$, equipped with a fixed 
invariant $(-1)^{k-1}$-Hermitian form. When $k$ is odd, we require that the discriminant of this $(-1)^{k-1}$-Hermitian form equals $1$.
Let 
\[
\udl{p} = [p_1^{m_1}, \ldots, p_r^{m_r}]
\]
be a partition of $\Fn=\dim_E V$, where $p_1>\cdots>p_r>0$ and $m_1,\ldots, m_r\geq 0$, subject to the parity conditions as in \cite{CM}.  Let $V_{p_i}$ be an $m_i$-dimensional $(-1)^{p_i-1}\epsilon$-Hermitian space over $E$, where $i=1,\ldots, r$. The datum
$
\left(\udl{p}, \{V_{p_i}\}\right)
$
is called a sesquilinear Young tableau of sign $\epsilon$. In particular if $r=1$, the sesquilinear Young tableau 
is simply written as 
\be \label{onepart}
([p_1^{m_1}], \{V_{p_1}\}) = (p_1^{m_1}, V_{p_1}).
\ee

If the $\epsilon$-Hermitian vector space $V$ of dimension $\Fn$ over $E$ has an isometry
\be \label{isom}
V\cong \bigoplus^r_{i=1} S_{p_i}\otimes V_{p_i}, \quad (\text{the orthogonal direct sum})
\ee
then we say the sesquilinear Young tableau $(\udl{p}, \{V_{p_i}\})$ is {\it admissible} for $V$. By \cite{GZ}, there is a one-to-one correspondence between the 
$F$-rational nilpotent orbits in $\Fg$ and the equivalence classes of admissible sesquilinear Young tableaux for $V$, 
except that when $G$ is an even special orthogonal group,  two $F$-rational nilpotent orbits may correspond to an admissible sesquilinear Young tableau with only even parts. See \cite[Section 2]{JLZ25} for more details.\footnote{A corrigendum: in \cite{JLZ25}, when $p_i$ is odd the sign $(-1)^{\frac{p_i-1}{2}}$ in (2.10) should be placed in the right hand side of (2.9) instead. Then the isometry (2.11) and all the subsequent calculations remain correct.}

If $\udl{p}' = [p_1^{m_1'}, \ldots, p_r^{m_r'}]$  and $\left(\udl{p}', \{V'_{p_i}\}\right)$ is 
an admissible sesquilinear Young tableau for another $\epsilon$-Hermitian space $V'$ over $E$, we define the direct sum
\be \label{directsum}
\left(\udl{p}, \{V_{p_i}\}\right) \oplus \left(\udl{p}', \{V'_{p_i}\}\right):= \left(\udl{p}\cup \udl{p}', \{V_{p_i}\oplus V'_{p_i}\}\right),
\ee
where $\udl{p}\cup \udl{p}' := [p_1^{m_1+m_1'},\ldots, p_r^{m_r+m_r'}]$. Clearly \eqref{directsum} is an admissible sesquilinear Young tableau for the $\epsilon$-Hermitian space $V\oplus V'$ (the orthogonal direct sum).  Thus the sesquilinear Young tableaux
of the same sign form a monoid, and the following result is clear by construction.

\begin{lem} \label{lem:tableau}
If  the $F$-rational nilpotent orbits $\CO$ and $\CO'$ in $\Fg(V)$ and $\Fg(V')$ correspond to the sesquilinear Young tableaux $(\udl{p}, \{V_{p_i}\})$ and $(\udl{p}', \{V_{p_i}'\})$ respectively, 
$f\in \CO$ and $f'\in \CO'$, 
then the $F$-rational nilpotent orbit  $\CO_{f+f'}$ in $\Fg(V\oplus V')$ corresponds to $\left(\udl{p}, \{V_{p_i}\}\right) \oplus \left(\udl{p}', \{V'_{p_i}\}\right)$.
\end{lem}

We recall the topological orders on $\CN_F(\Fg)_\circ$ following \cite[Definition 5.9]{JLZ25}.

\begin{defn} \label{defn:order}
For an $F$-rational orbit $\CO \in \CN_F(\Fg)_\circ$, denote by 
\[
\CO^\st:=\set{x\in \CN_F(\Fg) | \text{there exists $g\in \BG(\overline{F})$ such that $\Ad(g)(x)\in \CO$}}
\]
the corresponding $F$-stable orbit. 
For two $F$-rational orbits 
$\CO_1, \CO_2 \in \CN_F(\Fg)_\circ$, we say that 
\begin{enumerate}
\item 
$\CO_1\leq_F \CO_2$ under the $F$-rational topological order if $\CO_1$ is contained in the Hausdorff closure of $\CO_2$.

\item 
$\CO_1 \leq_\st \CO_2$ under the $F$-stable topological order if $\CO_1^\st$ is contained in the Zariski closure of $\CO_2^\st$.
\end{enumerate}
\end{defn}

\subsection{Generalized Whittaker models and algebraic wavefront set} \label{sec:GW}

Fix a non-trivial additive character $\psi_F$ of $F$, and  denote by $\kappa$ the Killing form of $\Fg$. Let $\CO$ be an $F$-rational nilpotent orbit in $\Fg$. Fix $f\in \CO$ and extend it to an $\frak{sl}_2$-triple $\{h, e, f\}$ in $\Fg$, where $h=[e,f]$ is semisimple with $[h,e]=2e$ and $[h,f]=-2f$. When $\CO=\{0\}$ it is understood that $h=0$. For $i\in \mathbb{Z}$,  put 
\be \label{grade}
\Fg_i:=\set{x\in \Fg | [h,x]=ix},\quad \Fg_{\geq i}:=\bigoplus_{j\geq i} \Fg_j\quad\text{and}\quad \Fg_{\leq i}:=\bigoplus_{j\leq i}\Fg_j.
\ee
Then we have a parabolic  subgroup $P=MU$ of $G$ with Levi subgroup $M$ and unipotent radical $U$, whose Lie algebras are $\frak m=\Fg_0$ and $\Fu= \Fg_{\geq 1}$ respectively. Moreover, if 
we set 
\[
V_{[i]}:=\set{v\in V | hv=iv},
\]
then we obtain that 
\be \label{levi}
\Fm= \Fg(V_{[0]})\oplus \bigoplus_{i>0} \frak{gl}(V_{[i]}).
\ee

Define a character $\psi_f$ of $U_2:=\exp(\Fg_{\geq 2})$ by 
\[
\psi_f(\exp(z)) := \psi_F(\kappa(f, z)),\quad z\in \Fg_{\geq 2}.
\]
where $\kappa(\cdot,\cdot)$ is the Killing form on $\Fg$. Following \cite[Section 3.3]{GZ}, 
\[
\kappa_f(x,y) := \kappa(f, [x,y]),\quad x, y\in \Fg_1
\]
defines a symplectic form on $\Fg_1$.  Let  $\omega_{\psi_f}$ be the irreducible unitarizable smooth representation of $U$ such that $U_2$ acts by the character $\psi_f$, which is unique up to isomorphism by the Stone-von Neumann theorem. Note that if $\Fg_1=\{0\}$, then $\omega_{\psi_f} = \psi_f$ is a character of $U=U_2$.  

When $\Fg_1\neq\{0\}$, fix a Lagrangian subspace $\Fl$ of $\Fg_1$ with respect to the symplectic form $\kappa_f$, and denote 
$U_{1.5}:=\exp(\Fg_{\geq 1.5})$, where $\Fg_{\geq 1.5}:=\Fl\oplus \Fg_{\geq 2}$. Then $\psi_f$ extends to a character of $U_{1.5}$ by making it trivial on the $\Fl$-part. By \cite[Lemma 2.4.3]{GGS17} we have that 
\[
\omega_{\psi_f}\cong \ind^U_{U_{1.5}}\psi_f.
\]
Here and thereafter, $\ind$ denotes the smooth compactly supported induction  (\cite{BZ})  when $F$ is non-archimedean, and 
the Schwartz induction (\cite{dCl}) when $F$ is archimedean. 

As before, denote by $\Pi_F(G)$ the set of equivalence classes of irreducible admissible representations of $G$, which are of the Casselman-Wallach type if $F$ is archimedean and are genuine if $G$ is metaplectic. 
For any $\pi\in\Pi_F(G)$, define the generalized Whittaker model of $\pi$ associated to $f\in\CO$ by
\[
\Wh_\CO(\pi) :=\Hom_{U}(\omega_{\psi_f}^\vee, \pi^*), 
\]
where  $\pi^*$ is the strong dual of $\pi$. By \cite[Lemma 2.5.2]{GGS17},
\[
\Wh_\CO(\pi) \cong  \Hom_{U_{1.5}}(\pi, \psi_f). 
\]
Define the twisted Jacquet module 
\be \label{Jac}
\CJ_\CO(\pi):= \left(\pi\,\widehat\otimes\, \omega_{\psi_f}^\vee\right)_U \cong \left(\pi\otimes \psi_f^{-1}\right)_{U_{1.5}},
\ee
where the subscripts indicate the maximal coinvariant Hausdorff quotients and the last isomorphism is given by \cite[Lemma 2.16]{GGS21}. Note that in the definition of $\Wh_\CO(\pi)$ and $\CJ_\CO(\pi)$, we may not require that $\pi$ is irreducible. 

\begin{defn} \label{algwf}
Assume that $\pi\in \Pi_F(G)$. The algebraic wavefront set of $\pi$ is 
\begin{eqnarray*}
\wf_\wm(\pi): =  \set{\CO\in \CN_F(\Fg)_\circ | \Wh_\CO(\pi)\neq \{0\}} =   \set{\CO\in \CN_F(\Fg)_\circ | \CJ_\CO(\pi)\neq \{0\}}.
\end{eqnarray*}
\end{defn}

\subsection{Descent and arithmetic wavefront sets} \label{sec:AW}
In \cite{JLZ25} we introduced the arithmetic wavefront set $\wf_\ari(\pi)$ for $\pi\in \Pi_F(G)$ with a 
generic local $L$-parameter. It is defined by means of consecutive descents of enhanced $L$-parameter 
of $\pi$, based on the local Langlands correspondence (established through \cite{Ar,KMSW, Mok, AGI+}) and the local Gan-Gross-Prasad conjecture (\cite{GGP1}) (which is now a theorem thanks to the works of many authors). We refer to \cite{JLZ25} for a list of relevant references. 
 
In this paper, we introduce a new wavefront set $\wf_\des(\pi)$ for each $\pi\in \Pi_F(G)$ with a 
generic local $L$-parameter by means of the local descents of $\pi$ and prove that $\wf_\des(\pi)=\wf_\ari(\pi)$ for each $\pi\in \Pi_F(G)$ with a generic local $L$-parameter (Theorem \ref{thm:des=ari}). The point is that 
$\wf_\des(\pi)$ is more suitable for our proof of Theorem \ref{thm:wfsc}. To this end, we put
\[
\CZ:=F^\times/ \set{ x\cdot c(x) | x\in E^\times},
\]
where $c$ is the generator of ${\rm Gal}(E/F)$. Let $0< p_1\leq \Fn$ such that $p_1$ is odd if $G$ is special orthogonal and is even if $G$ is symplectic or metaplectic, and finally there is no
parity constraint on $p_1$ if $G$ is unitary. Suppose that $\CO$ is an $F$-rational nilpotent orbit  in $\Fg$ corresponding to the partition $[p_1, 1^{\Fn-p_1}]$. 
Then we have two cases.

\begin{itemize}
\item
If $p_1>1$ and $\CO$ corresponds to the (unique) admissible sesquilinear Young tableau 
\[
([p_1, 1^{\Fn-p_1}], \{V_{p_1}, V_1\})
\]
such that the one-dimensional $(-1)^{p_1-1}\epsilon$-Hermitian space $V_{p_1}$ has discriminant $a_1\in \CZ$, we write 
$\CO=\CO_{p_1, a_1}$. As in \eqref{isom}, we may assume that 
\[
V=V'\oplus V'':= S_{p_1}\otimes V_{p_1}\oplus S_1\otimes V_1
\]
under the action of the $\frak{sl}_2$-triple $\{h, e, f\}$, where $f\in \CO_{p_1, a_1}$.  
Then $\CJ_{\CO_{p_1, a_1}}(\pi)$ in \eqref{Jac} is a smooth (and nuclear Fr\'echet of moderate growth if $F$ is archimedean) representation of 
\[
H:=\begin{cases} G(V''), & \text{if $G$ is not symplectic}, \\
\Mp(V''), & \text{if $G$ is symplectic},
\end{cases}
\]
where in the case that $G$ is unitary and $p_1$ is even, we take a splitting 
\[
G(V'') \hookrightarrow \Mp(\Res_{E/F}(\bm{\delta}^{(\epsilon+1)/2}\cdot V''))
\]
as in \cite[Section 12]{GGP1}. Note that 
the Lie algebra of $H$ is $\Fh=\Fm^f$, where $\Fm$ as in \eqref{levi} is the Lie algebra of the corresponding Levi subgroup $M$.
\item
If $p_1=1$, then $\CO=\{0\}$. In this case, if $V'$ is an anisotropic line in $V$ with discriminant $a_1\in \CZ$ and 
$V=V'\oplus V''$ is the orthogonal decomposition, we write 
\[
\CJ_{\CO_{1,a_1}}(\pi):=\pi|_H,\quad \text{where}\quad H:= G(V'').
\]
In this case we define the sesquilinear Young tableau $(p_1, V_{p_1}):=(1, V')$.
\end{itemize}
It is clear that $(G,H)$ form a relevant pair in the setting of the local Gan-Gross-Prasad conjecture (\cite{GGP1}). 

\begin{defn}[Descent Wavefront Set] \label{dfn:DWS}
For any $\pi\in\Pi_F(G)$ with a generic $L$-parameter, 
an $L$-descent orbit of $\pi$ is defined inductively to be an  $F$-rational nilpotent orbit 
\[
(p_1, V_{p_1}) \oplus \cdots \oplus (p_r, V_{p_r})\in \CN_F(\Fg)_\circ
\]
given  by  \eqref{directsum}, such that 
\begin{enumerate}
    \item $\CJ_{\CO_{p_1,a_1}}(\pi)$ is nonzero and admits $\sigma\in \Pi_F(H)$ as an irreducible quotient representation, i.e., a descent of $\pi$ along  the sesquilinear Young tableau  $(p_1, V_{p_1})$ in \eqref{onepart};
    \item \label{item:generic} the $L$-parameter of $\sigma$ is generic;
    \item $(p_2, V_{p_2}) \oplus \cdots \oplus (p_r, V_{p_r})$ is an $L$-descent orbit of $\sig$;
    \item \label{item:decreasing} $p_1\geq p_2$.
\end{enumerate}
The {\bf descent wavefront set} $\wf_\des(\pi)$ is the set consisting of all the $L$-descent orbits of $\pi$. 
\end{defn}
Note that if $p_1=1$, the $L$-descent orbit is the zero nilpotent orbit.
In view of the non-tempered Gan-Gross-Prasad conjecture (\cite{GGP2}), the above approach may be also applied to study the wavefront sets of Arthur type representations by relaxing the conditions \eqref{item:generic} and \eqref{item:decreasing} in Definition \ref{dfn:DWS}.

\begin{thm} \label{thm:des=ari}
    For any $\pi\in\Pi_F(G)$ with a generic $L$-parameter, the two wavefront sets $\wf_\des(\pi)$ and $\wf_\ari(\pi)$ are equal, i.e. 
    \[
    \wf_\des(\pi)=\wf_\ari(\pi).
    \]
\end{thm}

  \begin{proof}
First, we recall the definition of the arithmetic wavefront set $\wf_{\ari}(\pi)$ introduced in \cite[Section 5.3]{JLZ25}. 
Fix a Whittaker datum $a$ for the quasi-split form of $G$, which determines the local Langlands correspondence for $\Pi_F(G)$ and associates to $\pi$ a unique enhanced $L$-parameter $(\varphi,\chi)$.
In \cite[(5.14)]{JLZ25}, with respect to the Whittaker datum $a$, we attach to $(\varphi,\chi)$ a set $\CY_a(\varphi,\chi)$ of admissible sesquilinear Young tableaux as follows.
The consecutive descent of enhanced $L$-parameters, defined by means of the distinguished characters that are prescribed by the local Gan--Gross--Prasad conjecture, produces a chain of enhanced $L$-parameters
\[
(\varphi,\chi)\rightarrow (\phi_1,\chi_1')\rightarrow \cdots \rightarrow (\phi_r,\chi_r')
\]
together with a decreasing sequence of rows $(p_1,V_{p_1}),\dots,(p_r,V_{p_r})$ (associated with their corresponding one-dimensional $(-1)^{p_i-1}\epsilon$-Hermitian spaces $V_{p_i}$). 
Combining these rows yields an admissible sesquilinear Young tableau in $\CY_a(\varphi,\chi)$ and an $F$-rational nilpotent orbit in $\CN_F(\Fg)_{\circ}$ via the ``direct-sum" construction \eqref{directsum}.  
Following the construction in \cite[Section 5.3]{JLZ25}, the arithmetic wavefront set $\wf_{\ari}(\pi)$ is precisely the set of $F$-rational nilpotent orbits arising from those tableaux in $\CY_a(\varphi,\chi)$ whose underlying partitions are decreasing (equivalently, $p_1\ge p_2\ge\cdots$). 

On the other hand, $\wf_{\des}(\pi)$ is defined via the existence of a nonzero Jacquet module $\CJ_{\CO_{p_1, a_1}}(\pi)$ admitting an irreducible quotient $\sigma \in \Pi_F(H)$. 
By the definition of $L$-descent orbits, $\sigma$ must have a generic $L$-parameter, denoted by $\phi_1$.
By the local Gan--Gross--Prasad conjecture and the results of \cite{CJLZ}, reformulated in the present notation, the first occurrence index and the corresponding quotient in the representation-theoretic descent agree with the first arithmetic descent of the enhanced parameter.
Hence, an irreducible quotient $\sigma$ with generic $L$-parameter $\phi_1$ occurs in $\CJ_{\CO_{p_1,a_1}}(\pi)$ if and only if $(\varphi,\chi)$ has a first-step descent $(\phi_1,\chi_1')$ with first row $(p_1,V_{p_1})$ in the sense described above. Moreover, the genericity of the parameter of $\sigma$ and the inductive continuation of the descent are exactly the conditions required for the consecutive descent on the parameter side to proceed.  
By the induction hypothesis applied to $\sigma$ (which has generic $L$-parameter $\phi_1$), the subsequent terms in the orbit decompositions for the descent wavefront set and the arithmetic wavefront set must also coincide.  

Therefore, an $F$-rational nilpotent orbit belongs to $\wf_{\des}(\pi)$ (i.e., it arises from an iterative chain of nonzero descents through the quotients with generic $L$-parameters, with $p_1\ge p_2$ at each step) if and only if it is realized by a decreasing admissible tableau in $\CY_a(\varphi,\chi)$ via the corresponding  $L$-parameter descents. 
This is precisely the defining criterion for membership in $\wf_{\ari}(\pi)$ and proves that $\wf_{\des}(\pi)=\wf_{\ari}(\pi)$.
\end{proof}

 
\section{Composition Law for Generalized Whittaker Models}\label{sec:CL}

The results in this section hold for all local fields $F$ of characteristic zero. We establish a root exchange lemma (Lemma \ref{lem:exchange}) for general affine algebraic groups over $F$, and as an application we prove a composition law (Theorem \ref{Thm:CL}) for generalized Whittaker models of classical groups.

\subsection{Root exchange lemma} \label{sec:REL}

In this section, we formulate a version of the so-called root exchange lemma (\cite[Lemma 2.2]{GRS1}, \cite[Lemma 7.1]{GRS3}), which works for all local fields of characteristic zero following the arguments in \cite{GGS17}, and then prove a generalization of it that will be a key technical ingredient in our proof of Theorem \ref{thm:wfsc} and will be useful in future applications.  

For a group $G$, if $x,y\in G$, write $[x,y]:=xyx^{-1}y^{-1}$ for their commutator. 

\begin{lem}[Root Exchange]\label{lem:exchange0}
Let $C\subset A$ be unipotent groups over $F$, and $\psi_C$ be a character of $C$. 
Assume that 
\begin{enumerate}
   \item $X$ and $Y$ are abelian unipotent subgroups of $A$, normalize $C$ and preserve $\psi_C$;
   \item $[X,Y]\subset C$;
   \item $A=D\rtimes Y= B\rtimes X$, where $D=XC$ and $B=CY$;
   \item The pairing $X\times Y\to \mathbb{C}^{\times}$ defined by $(x,y)\mapsto \psi_C([x,y])$ is non-degenerate.
\end{enumerate}
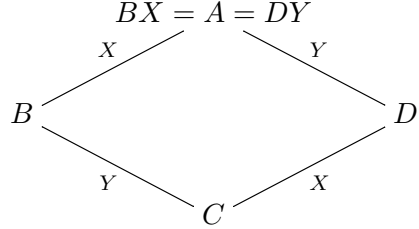
\[
\xymatrix{
& BX= A= DY \ar@{-}[dl]_X \ar@{-}[dr]^Y & \\
B & & D \\
& C \ar@{-}[ul]^Y \ar@{-}[ur]_X & 
}
\]
Let $\pi$ be a smooth representation of $A$, which is nuclear Fr\'echet of moderate growth when $F$ is archimedean. Then there is an isomorphism 
\[
\Hom_B(\pi, \psi_B) \cong \Hom_D(\pi, \psi_D),
\]
where $\psi_B$ and $\psi_D$ are the trivial extensions of $\psi_C$ to $B$ and $D$, respectively. 
\end{lem}

\begin{proof}
 By \cite[Lemma 3.2.1]{GGS17} we have an $A$-isomorphism $\ind^A_B(\psi_B^{-1})\cong \ind^A_D(\psi_D^{-1})$, hence by \cite[Lemma 2.5.2]{GGS17}
    \[
    \Hom_B(\pi, \psi_B) \cong \Hom_A(\ind^A_B(\psi_B^{-1}), \pi^*)\cong \Hom_A(\ind^A_D(\psi_D^{-1}), \pi^*) \cong \Hom_D(\pi, \psi_D),
    \]
    where $\pi^*$ is the strong dual of $\pi$.
\end{proof}

Applying Lemma \ref{lem:exchange0} repeatedly, we have the following useful generalization, a global analog of which can be found in \cite{JL16}. 

\begin{lem}[Generalized Root Exchange]\label{lem:exchange}
Let $C$ be a unipotent subgroup of an affine algebraic group $G$ over $F$, and let $\psi_C$ be a character of $C$. 
Assume that 
\begin{enumerate}
   \item \label{RE1} $X=\prod_{i=1}^mX_i$ and $Y=\prod_{i=1}^m Y_i$ are direct products of abelian unipotent groups contained in $G$ such that $X\cap C= Y\cap C=\{1\}$;
   \item \label{RE2} $[X_i,Y_j]\subset C$ for all $1\leq i \leq j\leq m$;
   \item \label{RE3} $[C, X_i]\subset CX_1X_2\cdots X_{i-1}$ and $[C, Y_i]\subset CY_{i+1}Y_{i+2}\cdots Y_m$ for all $1\leq i\leq m$; 
   \item \label{RE4} $X_i$ and $Y_i$  preserve $\psi_{<i}$ and $\psi_{>i}$ respectively, where $\psi_{<i}$ and $\psi_{>i}$ are the trivial extensions of $\psi_C$ to $CX_1X_2\cdots X_{i-1}$ and $CY_{i+1}Y_{i+2}\cdots Y_m$, for all $1\leq i\leq m$;
   \item \label{RE5} The pairing $X_i\times Y_i\to \mathbb{C}^{\times}$ defined by $(x,y)\mapsto \psi_C([x,y])$ is non-degenerate for all $1\leq i\leq m$.
\end{enumerate}
Let $\pi$ be a smooth representation of $G$, which is nuclear Fr\'echet of moderate growth when $F$ is archimedean. Then there is an isomorphism 
\[
\Hom_B(\pi, \psi_B) \cong \Hom_D(\pi, \psi_D),
\]
where $B=X C$ and $D=CY$, with characters $\psi_B$ and $\psi_D$ being the trivial extensions of $\psi_C$ respectively. 
\end{lem}

\begin{proof}
Denote $C_i=X_1X_2\cdots X_{i-1} C Y_{i+1}\cdots Y_m$ for $1\leq i\leq m$.
By the conditions \eqref{RE1}--\eqref{RE5}, for each $1\leq i\leq m$ the triple $(X_i,Y_i,C_i)$  satisfies the conditions in Lemma \ref{lem:exchange0}, which also allows us to extend $\psi_C$ to a character of $C_i$ trivially on $X_1X_2\cdots X_{i-1}\cdot Y_{i+1}Y_{i+2}\cdots Y_m$. We denote the trivial extension of $\psi_C$ to $C_i$ by $\psi_i$.

Applying Lemma \ref{lem:exchange0} for 
\[
\xymatrix{
  X_i C_i  \ar@{-}[dr]_{X_i} & & \ar@{-}[dl]^{Y_i} C_{i} Y_i  \\
   & C_i &
}
\]
we  obtain  that
\begin{equation}\label{eq:exch-individual-step-t}
 \Hom_{X_iC_i}(\pi, \psi_i)\cong  \Hom_{C_{i}Y_i}(\pi, \psi_{i-1 }),  
\end{equation} 
where by abuse of notation we still write $\psi_i$ for its trivial extension to $X_iC_i$ and $\psi_0:=\psi_D$, noting that 
$X_iC_i=C_{i+1}Y_{i+1}$ due to the conditions \eqref{RE1}--\eqref{RE4}. 

We have that $B=X_m C_m$ and $D=C_1Y_1$. In particular, $\psi_B = \psi_m$ and $\psi_D= \psi_0$. Applying \eqref{eq:exch-individual-step-t} inductively for $i=m,m-1,\dots, 1$, we obtain that
\[
\Hom_{B}(\pi, \psi_B)\cong \Hom_{X_{m-1}C_{m-1}}(\pi,\psi_{m-1}) \cong \cdots\cong  \Hom_{X_1C_1}(\pi,\psi_1)\cong \Hom_{D}(\pi,\psi_D)
\]
which can be illustrated in the following diagram:

{\small
\begin{equation*}
\begin{array}{c}
\xymatrix{
   B=X_m C_m \ar@{-}[dr]_{X_m} & & \ar@{-}[dl]^{Y_m} C_{m} Y_m=X_{m-1}C_{m-1}   \ar@{-}[dr]_{X_{m-1}} & &\ar@{-}[dl]^{Y_{m-1}}C_{m-1}Y_{m-1}=X_{m-2} C_{m-2}  \ar@{-}[dr]_{X_{m-2}}  & \cdots    \\
   & C_m & &C_{m-1} & &\cdots
   }\\
\xymatrix{
\cdots & C_2Y_2 = X_1 C_1 \ar@{-}[dr]_{X_1} \ar@{-}[dl]^{Y_2} & & \ar@{-}[dl]^{Y_1} C_{1}Y_1=D   \\
\cdots & & C_1  &
}
\end{array}
\end{equation*}
}
\end{proof}

\subsection{Composition law} \label{ssec:CL}

In this section, by using the root exchange lemma (Lemma \ref{lem:exchange}), we prove a composition law for generalized Whittaker models of classical groups over any local field of characteristic zero.
We mention that a special case of its global analog can be found in \cite{GRS2}. Our proof is Lie algebra theoretical, which avoids matrix computations and should be more conceptual. 

Assume that 
$\CO'\in \CN_F(\Fg)_\circ$ corresponds to a sesquilinear Young tableau of the form 
\[
([p_1^{m_1},\ldots, p_r^{m_r}, 1^{\Fn''}], \{V'_{p_i}, V'_1\}),
\]
where  $p_1> \cdots > p_r>1$, such that
\[
V  = \bigoplus^r_{i=1} S_{p_i}\otimes V'_{p_i} \oplus S_1\otimes V'_1
\]
under the action of an $\frak{sl}_2$-triple $\{h',e',f'\}\subset \Fg$ with $f'\in \CO'$. 
 Write $\udl{p}:=[p_1^{m_1},\ldots, p_r^{m_r}]$. Then $\Fn=\Fn'+\Fn''$, where $\Fn' =|\udl{p}|:= \sum^r_{i=1}m_i p_i$.
Put
\[
V':= \bigoplus^r_{i=1} S_{p_i}\otimes V'_{p_i} \quad \text{and}\quad V'':=S_1\otimes V'_1. 
\]
There is an involution $\theta$ of $\Fg$ such that
\be \label{inv}
\Fg^\theta = \Fg'\oplus \Fg'':=\Fg(V')\oplus \Fg(V''),   
\ee
and 
\be \label{inv-}
\Fg^{-\theta}\cong V'\otimes V''\text{ as $(\frak g'\oplus \frak g'')$-modules}.
\ee

Further assume that  $\CO'' \in \CN_F(\Fg'')_\circ$ corresponds to a sesquilinear Young tableau 
\[
(\udl{q}, \{V''_{q_j}\}),\quad \text{where}\quad \udl{q}=[q_1^{n_1}, \ldots, q_s^{n_s}], \ q_1>\cdots>q_s >0
\]
such that  
\[
V''= \bigoplus^s_{j=1} S_{q_j}\otimes V''_{q_j}
\]
under the action of an $\frak{sl}_2$-triple $\{h'', e'', f''\}\subset \Fg''$ with $f''\in\CO''$. 

The $\frak{sl}_2$-triples $\{h,e,f\}, \{h',e',f'\}$ and $\{h'',e'',f''\}$ are related as in the following lemma. 

\begin{lem}
We have an $\frak{sl}_2$-triple 
\[
\{h,e,f\}:=\{h'+h'', e'+e'', f'+f''\}.
\]
Let $\CO=\CO_f\in \CN_F(\Fg)_\circ$. Then $\CO$ corresponds to the sesquilinear Young tableau 
\[
(\udl{p}, \{V'_{p_i}\}) \oplus (\udl{q}, \{V''_{q_j}\})
\]
given by \eqref{directsum}.
\end{lem}

\begin{proof}
Since the $\frak{sl}_2$-triple $\{h',e',f'\}$ acts trivially on $V''$, it commutes with $\Fg''$, which implies the first assertion of the lemma. The second assertion follows
from Lemma \ref{lem:tableau}.
\end{proof}

Now we impose the assumption that 
\be \label{pq}
p_r \geq q_1>1. 
\ee
The main result of this section is the following composition law for generalized Whittaker models. 

\begin{thm}[Composition Law] \label{Thm:CL}
Let the notation and assumptions be as above. Let $\pi$ be a smooth representation of $G$, which is nuclear Fr\'echet of moderate growth if $F$ is archimedean, and is genuine if $G$ is metaplectic. Then 
\[
\Wh_\CO(\pi) \cong \Wh_{\CO''}(\CJ_{\CO'}(\pi)).
\]
\end{thm}

In the rest of this section, we  prove  Theorem \ref{Thm:CL} and give  Corollary \ref{cor:wf}.  
Keep the notation in \eqref{grade} and \eqref{inv} and write 
\[
\Fg'_i:=\set{ x\in \Fg | [h',x]=ix}\quad{\rm and}\quad \Fg''_i:=\set{x\in \Fg'' | [h'',x]=ix}, \quad {\rm for}\ i\in \BZ.
\]
Similar to \eqref{grade}, define $\Fg'_{\geq i}, \Fg'_{\leq i}, \Fg''_{\geq i}$ and $\Fg''_{\leq i}$ in the obvious way. We may fix an $h''$-stable Lagrangian subspace $\Fl' \subset \Fg'_1$ such that
$\Fg'_1\cap \Fg_{\geq 2}\subset \Fl'$, and also fix a Lagrangian subspace $\Fl'' \subset \Fg_1''$. Put 
\[
\Fg'_{\geq 1.5}:=\Fl' \oplus \Fg'_{\geq 2},\qquad \Fg''_{\geq 1.5}:=\Fl''\oplus \Fg''_{\geq 2}.
\]

We introduce
\begin{align*}
\Fc:&= \Fg''_{\geq 1.5}\oplus (\Fg_{\geq  2}\cap \Fg'_{\geq 2})\oplus \Fl' \\
& =\left ( (\Fg'_{\geq 1.5}\oplus \Fg''_{\geq 1.5})\cap \Fg^\theta \right) \oplus \left(\Fg_{\geq 2}\cap \Fg'_{\geq 2} \cap \Fg^{-\theta}\right)\oplus \Fl',
\end{align*}
where $\Fg^\theta$ and $\Fg^{-\theta}$ are as in \eqref{inv} and \eqref{inv-} respectively. For $1\leq i\leq p_1-2$, define
\begin{align*}
\Fx_i:& =\Fg_{2-i}\cap \Fg'_{\geq 2}\cap \Fg^{-\theta}, \\
\Fy_i:&=\Fg_{i}\cap\Fg'_{\leq 0}\cap\Fg^{-\theta}, \\
\Fc_i:& =\bigoplus_{k=1}^{i-1}\Fx_k\oplus \Fc \oplus \bigoplus_{k=i+1}^{p_1-2}\Fy_k.
\end{align*}
Put $C= \exp(\Fc)$, $X_i =\exp(\Fx_i)$, $Y_i=\exp(\Fy_i)$, and $\psi_C:=\psi_f|_C$. 

We shall verify that $C, X_i$, $Y_i$ and $\psi_C$ satisfy the conditions in Lemma \ref{lem:exchange}. It is straightforward to verify condition \eqref{RE1}. By the above definitions, for $ 1\leq i\leq p_1-2$ we have
\begin{align*}
& [\Fx_i,\Fy_j]\subset \Fg_{j-i+2}\cap \Fg^\theta \subset \Fc \qquad \text{ if }  1\leq i\leq j\leq p_1-2,\\ 
& [\Fx_i, \Fc]=[\Fx_i, (\Fg'_{\geq 1.5}\oplus \Fg''_{\geq 1.5})\cap \Fg^\theta]\subset \Fg_{\geq 3-i}\cap \Fg'_{\geq 2}\cap \Fg^{-\theta} \subset \Fc\oplus \bigoplus_{k<i}\Fx_k, \\ 
&[\Fy_i, \Fc]\subset \left(\Fg_{\geq i+1}\cap \Fg^{-\theta}\right)\oplus \left(\Fg_{\geq i+2}\cap \Fg^\theta\right) \subset \Fc\oplus \bigoplus_{k>i}\Fy_k.
\end{align*}
This validates conditions \eqref{RE2}, \eqref{RE3} and \eqref{RE4}. 

For a subspace $\Fs\subset \Fg$, denote by ${}^c \Fs$ the image of $\Fs$ under the Chevalley involution on $\Fg$. Then condition \eqref{RE5}  is immediate from the following lemma. 

\begin{lem} For $1\leq i\leq p_1-2$, the map $\ad \,f:\Fg\to\Fg$ restricts to an isomorphism  
\[
\ad\, f|_{\Fx_i}: \Fx_i \to {}^c \Fy_i,\quad x\mapsto [f,x].
\]
\end{lem}

\begin{proof} 
It is straightforward to show that $\ad\, f|_{\Fx_i}$ is injective. 
Indeed, take any elements $x\in \Fg_j'$ and $u\in \Fg'_{\geq j+1}$ with $j\geq 1$ and $x\neq 0$. Then $[f',x] \in \Fg_{j-2}'$ and is nonzero because 
$x$ is not a lowest weight vector under the action of the $\frak{sl}_2$-triple $\{h',e',f'\}$. From $[f',u]\in \Fg'_{\geq j-1}$ and 
$[f'',x+u]\in \Fg'_{\geq j}$, we conclude that 
\[
[f,x+u]=[f',x]+[f',u]+[f'',x+u]\neq 0.
\]
This proves the injectivity of $\ad f|_{\Fx_i}$. 

Consider the decompositions 
\[
\Fx_i = (\Fx_i \cap \Fg'_{\geq 3}) \oplus (\Fx_i\cap \Fg_2')\quad \text{and} \quad  \Fy_i=(\Fy_i\cap \Fg'_{\leq -1})\oplus (\Fy_i\cap \Fg'_{0}).
\]
We first prove that $\ad\, f$ restricts to an isomorphism 
\be \label{iso1}
\ad\,f: \Fx_i \cap \Fg'_{\geq 3} \to {}^c(\Fy_i\cap \Fg'_{\leq -1}). 
\ee
The map \eqref{iso1} is well-defined, noting that
\[
[f, \Fx_i\cap \Fg'_{\geq 3}] = [f, \Fg_{2-i}\cap \Fg'_{\geq 3}\cap \Fg^{-\theta}] \subset \Fg_{-i}\cap \Fg'_{\geq 1}\cap\Fg^{-\theta} = {}^c(\Fy_i\cap \Fg'_{\leq -1}).
\]
By a similar argument as above, the map 
\[
\ad\, e:  \Fg_{-i}\cap \Fg'_{\geq 1}\cap \Fg'_{\leq p_r-2}\cap \Fg^{-\theta} \to \Fx_i\cap \Fg'_{\geq 3}
\]
is also injective. Due to the assumption $p_r\geq q_1$ in \eqref{pq}, we have 
\[
\Fg_{-i}\cap \Fg'_{\geq 1}\cap \Fg^{-\theta}\subset \Fg'_{\leq p_r-2}\quad \text{for}\quad  i\geq 1.
\]
It follows that the map
\[
\ad\,e: {}^c(\Fy_i\cap \Fg'_{\leq -1})= \Fg_{-i}\cap \Fg'_{\geq 1}\cap \Fg^{-\theta} \to \Fx_i\cap \Fg'_{\geq 3}
\]
is injective. Hence \eqref{iso1} is an isomorphism. Moreover, the map
\be \label{iso2}
\ad\,f: \Fx_i\cap \Fg'_{2}\to 
{}^c(\Fy_i\cap \Fg'_0)
\ee
is also an isomorphism. 

Combining the isomorphisms \eqref{iso1} and \eqref{iso2} proves the lemma. 
\end{proof}

It is easy to see that 
\[
\Fc\oplus \bigoplus^{p_1-2}_{i=1} \Fx_i = \Fg'_{\geq 1.5}\oplus \Fg''_{\geq 1.5}
\quad
\text{and}
\quad
\Fc\oplus \bigoplus^{p_1-2}_{i=1}\Fy_i = \Fl \oplus \Fg_{\geq 2}=:\Fg_{\geq 1.5} 
\]
for a Lagrangian subspace $\Fl\subset \Fg_1$. 
Applying the root exchange Lemma \ref{lem:exchange}, we obtain that 
\[
\Wh_{\CO''}(\CJ_{\CO'}(\pi))\cong \Hom_{U'_{1.5}U''_{1.5}}(\pi, \psi_{f'}\psi_{f''})\cong \Hom_{U_{1.5}}(\pi, \psi_f)\cong \Wh_\CO(\pi),
\]
where $U'_{1.5} :=\exp(\Fg'_{\geq 1.5})$ and $U''_{1.5}:=\exp(\Fg''_{\geq 1.5})$. This proves Theorem \ref{Thm:CL}.

\

By induction, we have the following consequence of Theorem \ref{Thm:CL}.

\begin{cor} \label{cor:wf}
Assume that $\pi\in \Pi_F(G)$ has a generic local $L$-parameter. Then 
\[
\wf_\des(\pi)\subset \wf_\wm(\pi). 
\]
\end{cor}

\begin{proof}
We prove the assertion by induction on $\Fn$, which can be verified easily if $\Fn\leq 2$. Let 
    \[
    \CO=(p_1, V_{p_1})\oplus (p_2, V_{p_2})\oplus \cdots \oplus (p_r, V_{p_r})\in \wf_\des(\pi)
    \]
    be an $L$-descent orbit of $\pi$ as in Definition \ref{dfn:DWS}, where $\sigma_1$ is a descent of $\pi$ along $(p_1, V_{p_1})$, i.e. $\sigma_1$ is a quotient representation 
    of $\CJ_{\CO_{p_1,a_1}}(\pi)$ with $a_1:=\disc \, V_{p_1}$. By definition, 
    \[
    \CO'':=(p_2, V_{p_2})\oplus \cdots \oplus (p_r, V_{p_r})\in \wf_\des(\sigma_1)
    \]
   is an $L$-descent orbit of $\sigma_1$. Thus by induction hypothesis, it holds that $\CO''\in \wf_\wm(\sigma_1)$, i.e. $\Wh_{\CO''}(\sigma_1)\neq \{0\}$. It follows that 
   \[
   \Wh_{\CO''}(\CJ_{\CO_{p_1,a_1}}(\pi))\neq \{0\},
   \]
   hence by Theorem \ref{Thm:CL} we obtain that $\Wh_\CO(\pi)\neq \{0\}$, i.e. $\CO\in \wf_\wm(\pi)$. This finishes the proof of the corollary.
\end{proof}

\section{Arithmetic Wavefront Set: the Archimedean Case} \label{sec:AWF}

We consider real classical groups in the rest of this paper. In this section, we recall the algorithm of \cite{JLZ25} that computes the maximal elements ${\rm WF}_{\rm ari}(\pi)^{\max}$ in terms of enhanced $L$-parameters. We refer to {\it loc. cit.} for more details. 

\subsection{Enhanced $L$-parameters} \label{sec:ELP}

Let $E =\BR$ or $\BC$, and write ${\rm Gal}(E/\mathbb{R})=\langle c \rangle$.  Let $V$ be an $\Fn$-dimensional  $\epsilon$-hermitian vector space over $E$ with $\epsilon = \pm 1$,  
i.e.  $V$ is equipped with a non-degenerate pairing
\[
\langle\, ,\,\rangle: V\times V\to E
\]
that is linear in the first variable, $c$-linear in the second variable, and satisfies
\[
\langle y, x\rangle = \epsilon\cdot c ( \langle x, y\rangle), \quad x, y\in V.
\]

Let $G$ be the group of real points of the connected algebraic group ${\rm Isom}(V)^\circ$,
or the metaplectic double cover of $\Sp(V)$ in the case that 
 $E=\mathbb{R}$ and $\epsilon=-1$. 
 Then 
\be \label{cg}
G=\RU(p,q),\quad \SO(p,q),\quad {\rm Sp}_{2n}(\mathbb{R}) \quad \textrm{or}\quad  {\rm Mp}_{2n}(\mathbb{R}).
\ee
Fix a quasi-split pure inner form $G^*$ of $G$. Note that $G^*$ is unique unless $G =\RU(p,q)$ with $p+q=2n+1$, or $G=\SO(p,q)$
with $p+q=2n$ and $p\equiv n+1 ({\rm mod}\ 2)$.
Set $\frak n = \dim_E V (=p+q \textrm{ or }2n)$. In the first two cases of \eqref{cg}, we also write $G=\RU_{\frak n}$ or $\SO_{\frak n}$ for convenience and for the consistency of notation with that in \cite{JLZ25}.
Define
\be \label{disc}
{\rm disc}(V):=(-1)^{\frac{(p-q)(p-q-1)}{2}} \in \{\pm 1\}.
\ee

The $L$-group of $G$ is ${}^LG=G^\vee\rtimes W_\BR$, where $G^\vee$ is the complex dual group of $G$ and $W_\mathbb R$ is the real Weil group.
An $L$-parameter for $G$ is an admissible homomorphism $W_\BR\to {}^LG$.
Following \cite{GGP1, JLZ25}, we realize $L$-parameters as $G$-relevant conjugate self-dual semisimple representations of $W_E$.

When $E=\mathbb{C}$, for $\alpha\in \BZ$ define the character 
\[
\varsigma_\alpha: W_\BC = \BC^\times \to \BC^\times, \quad z\mapsto \left(|z|/z \right)^\alpha.
\]
When $E=\mathbb{R}$, for $\alpha \in \BZ$ define the two-dimensional representation 
$
\sigma_\alpha:={\rm Ind}^{W_\mathbb{R}}_{W_\mathbb{C}}\,\varsigma_\alpha
$
of $W_\mathbb{R}$. 
Then $\sigma_\alpha\cong \sigma_{-\alpha}$  is irreducible when $\alpha\neq 0$, and 
$
\sigma_0 \cong \mathbbm{1}\oplus \sgn,
$
where $\mathbbm{1}$ and ${\rm sgn}$ denote the trivial and sign characters of $\mathbb{R}^\times$ respectively, viewed as characters of $W_\mathbb{R}$ via local class field theory. 
Set
\[
\rho_\alpha:=\begin{cases} \varsigma_\alpha, & \textrm{if }E=\mathbb{C}, \\
\sigma_\alpha, & \textrm{if }E=\mathbb{R}.
\end{cases}
\]
 
 A generic $L$-parameter of $G^*$ can be viewed as a conjugate self-dual semisimple representation of  $W_E$ of the form 
\be \label{varphi}
\varphi:=\tau \oplus \varphi_{\rm gp} \oplus {}^c\tau^\vee
\ee
where ${}^c\tau^\vee$ is the conjugate dual of $\tau$, and $\varphi_{\rm gp}$ is the direct sum of all irreducible conjugate self-dual summands of $\varphi$ having the same parity as $\varphi$.

More precisely, 
\be \label{gp}
\varphi_{\rm gp} =  m_1\rho_{\alpha_1}\oplus m_2 \rho_{\alpha_2}\oplus \cdots\oplus m_r\rho_{\alpha_r} \oplus \varphi_{\rm quad},
\ee
where 
\begin{itemize}
\item $m_1, \ldots, m_r$ are positive integers;
\item $\alpha_1<\alpha_2<\cdots <\alpha_r$ are integers satisfying the following conditions:  they are nonnegative when $E=\mathbb{R}$; 
even when $G^\vee$ is orthogonal; 
odd when $G^\vee$ is symplectic, and of the same parity as ${\frak n}-1$ when $G^* =\RU_{\frak n}$;
\item $\varphi_{\rm quad} = \sgn^{m_1+\cdots+m_r}$ when $G^* ={\rm Sp}_{2n}(\mathbb{R})$ and zero otherwise;
\item $\det\varphi(-1) = {\rm disc}(V)$ when $G^*=\SO_{2n}$;
\item $\dim \varphi$ is the dimension of the natural representation of $G^\vee$.
\end{itemize}
For convenience, in the unitary group case, if we drop the assumption that $\varphi$ has parity $(-1)^{{\frak n}-1}$ in the above definition, then we refer to $\varphi$ as a {\it modified $L$-parameter}.

We may assume that $\tau$ in \eqref{varphi} is of the form
\be \label{tau}
\tau = |\cdot |_E^{s_1} \tau_1 \oplus \cdots \oplus |\cdot |^{s_l}_E \tau_l,
\ee
where $\tau_1,\ldots, \tau_l$ are irreducible and bounded, and $s_1\geq \cdots \geq s_l\geq 0$. 

Associated to $\varphi$ is the component group $\CS_\varphi \cong \CS_{\varphi_{\rm gp}}\cong \BZ_2^r$, and an {\it enhanced $L$-parameter} is a pair $(\varphi,\chi)$ where $\chi\in \widehat{\CS_\varphi}$. 
Thus $\chi$ is determined by a sequence of signs 
\begin{equation} \label{eq:chi-alpha}
\chi(\alpha_1), \ldots, \chi(\alpha_r).
\end{equation} 
Fixing the Whittaker datum of $G^*$ as in \cite[Appendix]{AA}, the local Langlands correspondence established through \cite{Ar, Mok, KMSW, AGI+} yields a bijection 
\[
\widehat{\CS_\varphi} \xrightarrow{\ \sim \ }\Pi_\varphi[G^*],  
\qquad \chi \longmapsto \pi(\varphi,\chi),
\]
where the Vogan packet $\Pi_\varphi[G^*]$ of $\varphi$ is a disjoint union of irreducible representations of pure inner forms of $G^*$.

\subsection{Arithmetic wavefront set} \label{ssec:AWF} 

For the proof of the main result, we will reduce to the case that $\pi$ is a discrete series representation; see Section \ref{sec:AWFAV}.
To keep notation minimal, we only explicate ${\rm WF}_{\rm ari}(\pi)^{\max}$ for the discrete $L$-parameters introduced in \cite{JLZ25}.
We refer to \cite[Section~5.1]{JLZ25} for the complete definition of ${\rm WF}_{\rm ari}(\pi)$.

Assume that $\varphi$ is a discrete  $L$-parameter of $G$, so that $\varphi = \varphi_{\rm gp}$ and is multiplicity-free, that is,
\begin{equation*}\label{eq:varphi-discrete}
\varphi = \rho_{\alpha_1} \oplus \cdots \oplus \rho_{\alpha_n} \oplus \varphi_{\rm quad},
\end{equation*}
where $\alpha_1<\cdots <\alpha_n$ and 
\[
n := \left\lfloor \dfrac{\frak n}{[\BC: E]}\right\rfloor \quad (\text{$\lfloor\cdot\rfloor$ denotes the integer part and ${\frak n}=\dim_{E}V$}),
\]
so that $n$ is the rank of $G^\vee$.
In particular $\CS_\varphi \cong \BZ_2^n$. 

In this case $G$ is of equal rank (see \cite[Lemma~1.6]{A}), excluding the pure inner class of $\SO(p,q)$ with $p,q$ odd.
For a discrete series $\pi=\pi(\varphi,\chi)$ of $G$, the set ${\rm WF}_{\rm ari}(\pi)^{\max}$ is a singleton consisting of a real nilpotent orbit
$\CO(\varphi,\chi)$ parametrized by a signed Young tableau.
Following \cite[Section~9]{JLZ25}, we write
\begin{equation}\label{eq:awf}
\CO(\varphi,\chi)=((p_1,\epsilon_1),\dots,(p_k,\epsilon_k)),
\end{equation}
where $p_1\ge \cdots \ge p_k$ is a partition of $\frak n$ and $\epsilon_i\in\{\pm\}$, subject to:
\begin{itemize}
\item
$p_i$'s are odd if $G= \SO(p,q)$ and even if $G = {\rm Sp}_{2n}(\BR)$ or ${\rm Mp}_{2n}(\BR)$;

\item each row $(p_i, \epsilon_i)$ consists of $p_i$ alternating signs ending with $\epsilon_i$; 

\item there are $p$ pluses and $q$ minuses in $\CO(\varphi,\chi)$ if $G=\RU(p, q)$ or $\SO(p,q)$. 
\end{itemize}

The algorithm of \cite{JLZ25} computes $\CO(\varphi,\chi)$ inductively via consecutive descents of enhanced $L$-parameters.
Motivated by the local Gan--Gross--Prasad conjecture (see \cite{GGP1}), we introduced the \emph{first descent} of $(\varphi,\chi)$, denoted $\CD_{p_1}(\varphi,\chi)$, which is a set of enhanced discrete (modified) $L$-parameters $(\phi,\chi')$ of opposite parity (\cite[Theorem 5.3]{JLZ25} and \cite[Theorem 1.5]{CJLZ}). The obtained set 
$\CD_{p_1}(\varphi,\chi)$ can be explicated as follows.

We first define the {\it sign alternating set} ${\rm sgn}(\chi)$. 
Define
\be \label{sgn}
{\rm sgn}(\chi):= \begin{cases} \set{ 1\leq i <n | \chi(\alpha_{i+1}) = - \chi(\alpha_i)}, & \textrm{if }G^* \neq \SO_{2n+1}, \\
 \set{ 0\leq i <n | \chi(\alpha_{i+1}) = - \chi(\alpha_i)}, &  \textrm{if }G^*  = \SO_{2n+1},
 \end{cases} 
\ee
where in the case $G^*=\SO_{2n+1}$ we set $\alpha_0:=-1$ and $\chi(\alpha_0):=1$.
Write
\[
s_\chi:= \#\sgn(\chi),
\]
and 
\[
\sgn(\chi)= \set{ i_1< i_2< \cdots < i_{s_\chi}}. 
\]
For convenience, we record the first row $(p_1,\epsilon_1)$ of $\CO(\varphi,\chi)$ and the elements $(\phi,\chi')\in \CD_{p_1}(\varphi,\chi)$
for each group type in the following form.
Define
\be \label{phi}
\phi = \begin{cases}
\rho_{\beta_1}\oplus \cdots \oplus \rho_{\beta_{s_\chi}} &\text{ if } G^*\ne \Mp_{2n}(\BR),\\ 
\rho_{\beta_1}\oplus \cdots \oplus \rho_{\beta_{s_\chi}} \oplus \sgn^{s_\chi}
 &\text{ if } G^*= \Mp_{2n}(\BR),
\end{cases}
\ee
with $\alpha_{i_j}<\beta_j <\alpha_{i_j+1}$ for all $1\leq j\leq s_\chi$.
The signs in $\chi'$ are the values $\chi'(\beta_j)$, arranged analogously to \eqref{eq:chi-alpha}. The relevant 
data for the first descent are summarized as follows:

\begin{table}[!htbp]
\begin{tabular}{cccc}
\hline 
$G^*$ & $p_1$ & $\epsilon_1$ & $\chi'$\\ \hline 
$\RU_n$& $ n- s_\chi$ & $\chi_n $ &$((-1)^{n - j + 1} \chi_n :  j \in \sgn(\chi))$\\ \hline
$\SO_{2n+1}$& $2(n-s_\chi)+1$ & ${\rm disc}(V)\chi_n$ & $((-1)^{j - 1} \delta_0 \chi_n :  j \in \sgn(\chi))$\\ \hline
$\SO_{2n}$& $2(n-s_\chi)-1$ & $\chi_n$ & $((-1)^j :  j \in \sgn(\chi))$\\ \hline
$\Sp_{2n},~\Mp_{2n}$& $2(n-s_\chi)$ & $\chi_n$ & $((-1)^{n - j + 1} \chi_n :  j \in \sgn(\chi))$\\ \hline
\end{tabular}

\quad

\caption{The first row and first descent}\label{tab:descent}
\end{table}
Here $\chi_n:=\chi(\alpha_n)$ and the last column $\chi'$ is called the {\it $L$-descent} of $\chi$.
In addition, to parametrize the pure inner class when $G^*=\SO_{2n+1}$, we introduce the auxiliary sign
\be \label{delta0}
\delta_0:= (-1)^n{\rm disc}(V) \qquad (\text{see }\eqref{disc}).
\ee 
The auxiliary sign for $\SO_{2s_\chi+1}$, the descent of $G^*=\SO_{2n}$, is given by
\[
\delta_0' =(-1)^{n-s_\chi-1}\chi_n.
\]  

Next, we apply the same procedure to $(\phi,\chi')$ to produce the second row $(p_2,\epsilon_2)$, which is independent of the choice of the $\beta_j$ in \eqref{phi}.
Finally, by induction, we obtain the signed Young tableau
\[
\CO(\varphi,\chi)=((p_1,\epsilon_1),(p_2,\epsilon_2),\dots,(p_k,\epsilon_k)).
\]
More detailed discussions on our algorithm will be given in the next section via examples. 

\subsection{Examples of $L$-parameter descent}
To illustrate the algorithm in Section~\ref{ssec:AWF}, we give examples for unitary and orthogonal groups.
More examples can be executed on the website \cite{Luo}.

\subsubsection{Unitary groups} \label{ssec:AW-U}

We first carry out the inductive algorithm for a discrete enhanced $L$-parameter $(\varphi,\chi)$ for $\RU(10,8)$.
Following the notation in \eqref{gp} and \eqref{eq:chi-alpha}, write
\[
\varphi=\rho_{\alpha_1} \oplus \cdots \oplus \rho_{\alpha_{18}} 
\]
where $\alpha_1<\alpha_2<\cdots<\alpha_{18}$ are odd integers, and
the character $\chi$ of $\CS_{\varphi}$ as
\[
-+-+--+-+--+-++---.
\]
First, we compute the first descent $\CD_{p_1}(\varphi,\chi)$ of $(\varphi,\chi)$ with $p_1=6$.
By \eqref{sgn}, we have 
\[
{\rm sgn}(\chi)=\{1,2,3,4,6,7,8,9,11,12,13,15\}.
\]
By Table \ref{tab:descent}, $\CD_{p_1}(\varphi,\chi)$ consists of 
modified $L$-parameters $(\phi,\chi')$ such that 
$\phi = \rho_{\beta_1}\oplus \cdots \oplus \rho_{\beta_{12}}$,
 where $\beta_j$ are even integers for all $1\leq j\leq 12$ satisfying that
\[
\begin{aligned}
& \alpha_1<\beta_1<\alpha_2<\beta_2<\alpha_3<\beta_3<\alpha_4<\beta_4<\alpha_5<\alpha_6<\beta_5<\alpha_7<\beta_6<\alpha_8 \\
<\, &\beta_7<\alpha_9<\beta_8<\alpha_{10}<\alpha_{11}<\beta_9<\alpha_{12}<\beta_{10}<\alpha_{13}<\beta_{11}<\alpha_{14}<\alpha_{15}<\beta_{12}<\alpha_{16},
\end{aligned}
\]
and the character $\chi'$ of $\CS_\phi$ is given by
\[
-+-++-+--+--.
\]
It is easy to note that the cardinality $\#\CD_{p_1}(\varphi,\chi)$ is $\prod_{j\in\sgn(\chi)} \left(\frac{\alpha_{j+1}-\alpha_j}{2}\right)$.

Continuing the process, we obtain the following table of consecutive descents:
\[
\begingroup
\setlength{\arraycolsep}{6pt}
\renewcommand{\arraystretch}{1.2}
\begin{array}{@{}lcl@{}}
\chi & \epsilon_1 & p \\
\hline
-+-+--+-+--+-++--- & & \\
-+-++-+--+-- & - & 6 = 18-12 \\
-+--+--+ & - & 4 = 12-8 \\
+--++ & + & 3 = 8-5 \\
-- & + & 3 = 5-2 \\
& - & 2
\end{array}
\endgroup
\]
In this table, the third column gives the row length $p_i$ of the signed Young diagram $\CO(\varphi,\chi)$; the second column gives the sign $\epsilon_i$, which equals the rightmost sign of the character being descended at each step; and the first column records the successive descended characters.

An alternative presentation is the following triangular array of signs:
\[
\begingroup
\setlength{\arraycolsep}{2pt}
\renewcommand{\arraystretch}{1.15}
\scriptsize
\begin{array}{@{}cccc cccc cccc cccc cccc ccccc@{}}
-&&+&&-&&+&&--&&+&&-&&+&&--&&+&&-&&++&&---\\
&-&&+&&-&&+&&+&&-&&+&&-&&-&&+&&-&&-&\\
&&-&&+&&-&&&&-&&+&&-&&&&-&&+&&&&\\
&&&+&&-&&&&&&-&&+&&&&&&+&&&&&\\
&&&&-&&&&&&&&-&&&&&&&&&&&&\\
\end{array}
\endgroup
\]
From Table~\ref{tab:descent}, we read off the signed Young tableau $\CO(\varphi,\chi)$ as
\[
\begin{ytableau} 
+&- & + & - & + & - \\
+ & - & + & - \\
+ & - & + \\
+ & - & + \\
+ & - \\
\end{ytableau}
\]

\subsubsection{Special orthogonal groups} \label{ssec:AW-SO}
Since the descent process for special orthogonal groups involves an auxiliary sign, we include one example to clarify this feature.
Consider a discrete enhanced $L$-parameter $(\varphi,\chi)$ for $\SO(16,15)$ with character
\[
-++-+-++-+--+++.	
\]
Since the group is $\SO(16,15)$, we have $\disc(V)=1$, hence the auxiliary sign in \eqref{delta0} equals $\delta_0=-1$.
If instead $\delta_0=1$, then this corresponds to $\SO(15,16)$, which is isomorphic to $\SO(16,15)$ but yields a different signed Young tableau.

We compute $\sgn(\chi)=\{0,1,3,4,5,6,8,9,10,12\}$ with $p_1=11$ and $\epsilon_1=+$. Then the character $\chi'$ of elements in the first descent is 
\[
+--+-++-++.
\]
Proceeding as above, we obtain the table of consecutive descents
\[
\begingroup
\setlength{\arraycolsep}{6pt}
\renewcommand{\arraystretch}{1.2}
\begin{array}{@{}lccc@{}}
\chi & \epsilon_1 & \delta_0 & p \\
\hline
-++-+-++-+--+++ &  & - & \\
+--+-++-++ & + &  & 11 \\
--+--+ & + & - & 7 \\
++-- & - &  & 5 \\
+ & - & - & 5 \\
 & + &  & 3
\end{array}
\endgroup
\]
and the signed Young diagram $\CO(\varphi,\chi)$
\[
\begin{ytableau} 
+&- & + & - & + & - & + & - & + & - & + \\
+&- & + & - & + & - & + \\
- & + & - & + & - \\
- & + & - & + & - \\
+ & - & +\\
\end{ytableau}
\]

\quad

\section{Explicit Local Langlands Correspondence}  \label{sec:LLC}

\def\ii{\mathrm{i}}
\def\Inn{\mathrm{Inn}}

In this section, we give the correspondence between discrete enhanced $L$-parameters and discrete series 
representations following \cite{A} and review the algorithm for the associated varieties of these modules following \cite{Tr05}. 
From now on, we fix a choice $\ii=\sqrt{-1}$, the square root of $-1$.

\subsection{Pure inner form and Vogan packet}
Let $G$ be as in \eqref{cg}, with Lie algebra $\frak{g}$; write $\CG$ and $\frak g_\BC$ for their complexifications respectively.   Let $Z_\CG$ be the center of $\CG$.
Fix an antiholomorphic involution $\varsigma$ on $\CG$ 
such that $\CG^\varsigma$ is a compact real form of $\CG$.
If $\theta$ is an algebraic involution on $\CG$ commuting with $\varsigma$, then $\CG^{\varsigma\circ \theta}$ is a real form of $\CG$ with $\theta$ as its Cartan involution.

Fix a maximal $\varsigma$-stable torus $T$ of $\CG$.
The Weyl group of $\CG^\varsigma$ (and of $\CG$) with respect to the Cartan subgroup $T^\varsigma$ is $W := N_{\CG^\varsigma}(T^\varsigma)/T^{\varsigma}$.
Denote by $(X^*(T),R, X_*(T),R^\vee)$ the root datum of $\CG$.
Take $G^\vee$ to be the complex dual group of $G$, defined by the dual root datum $(X_*(T),R^\vee,X^*(T),R)$ when $G$ is linear, and the complex symplectic group when $G$ is the metaplectic group.
Let $T^\vee$ be the dual torus of $T$ in $G^\vee$ so that $X^*(T^\vee)=X_*(T)$ and $X_*(T^\vee)=X^*(T)$.

The exponential map
\begin{equation}\label{eq:exp}
\exp: X_*(T)\otimes_{\BZ} \BC \to T(\BC)
\end{equation}
sends $\chi \otimes_{\BZ} a \mapsto \chi(\exp(a))$, where $(\chi : \BC^\times \to T) \in X_*(T)$ and $a\in \BC$.

Fix a Borel subgroup $B^\vee$ of $G^\vee$ containing $T^\vee$, and let $\rho^\vee \in X^*(T^\vee)\otimes_\BZ \BR = X_*(T)\otimes_\BZ \BR$ be the half-sum of the positive roots in $B^\vee$, and set
\[
x_b := \exp(\pi \ii \rho^\vee)\in T(\BC)
\]
as the ``base point'' for what follows.
Since $z_b := x_b^2 \in Z_\CG$, the conjugation by $x_b$ defines an involution on $\CG$.
In fact, the conjugation by $x_b$ is a Cartan involution of the quasi-split real form $\CG^{\varsigma \circ \Inn(x_b)}$ of $\CG$.

Define
\begin{equation}\label{eq:CX}
\CX := \set{x\in T | x^2 = z_b},
\end{equation}
which carries a natural $W$-action by conjugation.
Because $z_b^2 = 1$, the set $\CX$ is contained in the maximal compact subgroup of $T$, namely $T^\varsigma$.

Suppose $x$ and $x'$ lie in the same $W$-conjugacy class.
Choose $\dot w \in N_{\CG^\varsigma}(T)$ with ${}^{\dot w} x = x'$.
Then conjugation by $\dot w$ induces an isomorphism of real reductive groups
\begin{equation}\label{eq:iota1}
\iota_{x,x'} := \Inn(\dot w) : \CG^{\varsigma \circ \Inn(x)} \longrightarrow \CG^{\varsigma \circ \Inn(x')}.
\end{equation}
If $\dot w'$ is another element in $N_{\CG^\varsigma}(T)$ satisfying ${}^{\dot w'} x = x'$, then
\begin{equation}\label{eq:conj-in-K}
\dot w ^{-1} \dot w' \in \CG^\varsigma \cap \CG^{\varsigma \circ \Inn(x')}.
\end{equation}
For linear groups we set
\[
G_{x} := \CG^{\varsigma \circ \Inn(x)} \text{ for } x \in \CX \quad \text{ and }  \quad G^* := G_{x_b}.
\]

Turning to the metaplectic case: let $G$ be the metaplectic double cover of $\CG^{\varsigma \circ \Inn(x_b)}$ (with projection map $\pr$).
Since $z_b = -1$, we have that $\CX$ is a single $W$-orbit.
For $x\in \CX$, fix a lift $\dot w_x \in N_{\CG^\varsigma}(T)$ such that ${}^{\dot w_x} x_b = x$, and define $G_{x}$ by the cartesian diagram
\begin{equation}\label{eq:iota2}
\begin{tikzcd}
G_{x} \ar[r,"\iota_{x,x_b}"]  \ar[d]& G_{x_b} \ar[r,"\cong"]\ar[d,"\pr"]& \Mp_{2n}(\BR) \\
\CG^{\varsigma \circ \Inn(x)} \ar[r, "\Inn(\dot w_x)"] & \CG^{\varsigma \circ \Inn(x_b)} \\
\end{tikzcd}
\end{equation}

In all cases, the \emph{pure inner forms} (also known as \emph{strong inner forms}) of $G$ are classified by the $W$-conjugacy classes of elements in $\CX$ (see \cite[Proposition~12.18]{AC} and \cite[Lemma~2.3]{A}).
Let
\begin{equation}\label{eq:strong-inner-forms}
\set{x_1,\cdots, x_k}
\end{equation}
be fixed representatives of the $W$-conjugacy classes of elements in $\CX$. 
Then \[\Set{G_{x_i} | i=1,2,\cdots, k}\] is the set of all pure inner forms of $G$ (up to isomorphism). 


The Vogan packets are defined as follows. 
For $x\in \CX$, let $\Pi_{\varphi}(G_x)$ be the $L$-packet consisting of the elements in $\Pi_\BR(G_x)$ with $L$-parameter $\varphi$; we will define this precisely for discrete series parameters in Section~\ref{sec:discrete}.

Suppose $x$ and $x'$ are in the same $W$-conjugacy class.
Let $\dot w \in N_{\CG^\varsigma}(T)$ be such that ${}^{\dot w} x = x'$ as before, and let $\iota_{x,x'}$ be defined by \eqref{eq:iota1} and its natural extension to the metaplectic case via \eqref{eq:iota2}.
We say
\[
(x,\pi) \sim (x',\pi') \text{ if and only if } \pi \cong \pi' \circ \iota_{x,x'}. 
\]
Thanks to \eqref{eq:conj-in-K}, the definition of $\sim$ is independent of the choice of $\dot w$.
As usual, we denote the equivalence class of $(x,\pi)$ under the relation $\sim$ by $[x,\pi]$.

The Vogan packet $\Pi_\varphi[G^*]$ is defined to be  
\begin{equation}\label{Vp}
    \Pi_\varphi[G^*] := \bigsqcup_{i=1}^k \Pi_{\varphi}(G_{x_i}) 
\end{equation}
for a fixed choice of the representatives \eqref{eq:strong-inner-forms}.
On the other hand, it is more convenient to work symmetrically, following the Atlas convention.
Under the equivalence relation $\sim$, we have the following natural identification
\[
\Pi_\varphi[G^*] = \left.\bigsqcup_{x\in \CX} \Pi_{\varphi}(G_x)\right\slash \sim. 
\]

\subsection{Discrete series representations} \label{sec:dsrepn} 
When $G$ is linear, consider the element 
\[z^\vee := \exp(2\pi \ii \rho)\]
in $Z^\vee$ (the center of $G^\vee$) where $\exp$ is defined in the same spirit as \eqref{eq:exp} and $\rho$ is the half-sum of the positive roots for a choice of positive root system of $\CG$.
The element $z^\vee$ is independent of any choice. It is given by the following formula in our case:
\begin{equation}\label{eq:z-vee}
z^\vee = \begin{cases}
	(-1)^{n-1} & \text{if $G^\vee$ is a general linear group},  \\
	-1 & \text{if $G^\vee$ is a symplectic group}, \\
	1 & \text{if $G^\vee$ is an orthogonal group}.
\end{cases} 
\end{equation}
For the metaplectic group, we define $z^\vee = -1$ by convention.  

Recall that an element $\lambda$  in $X^*(T)\otimes \BR$ is called regular if  
\[
\inn{\lambda}{\alpha^\vee}\neq 0 \text{ for all } \alpha^\vee \in R^\vee.
\]
Now define the positive system $\Delta_\lambda^+$ by 
\[
\Delta^+_\lambda := \Set{\alpha \in R | \inn{\lambda}{\alpha^\vee} > 0}.
\]
Let $\rho_\lambda$ be the half-sum of the positive roots for the positive system $\Delta^+_\lambda$ and 
$\fbb_\lambda$ be the Borel subalgebra of $\fgg_\BC$ corresponding to $\Delta^+_\lambda$.
For $x\in \CX$, let $K_x := G_{x}^{\Inn(x)}$ be the maximal compact subgroup of $G_{x}$ when $G$ is linear and the preimage of $G^{\Inn(x)}_x$ in $G_{x}$ when $G$ is metaplectic.
Let 
\[
\rho_{x,\lambda} := \frac{1}{2} \sum_{\alpha \in \Delta^+_\lambda, \alpha(x) = 1} \alpha
\]  
be the half-sum of the positive roots in the complexified Lie algebra of $K_x$.

\begin{thm}[{\cite[Theorem~5.3]{VZ}}]\label{def:discrete-series}
 Suppose $\lambda$ is a regular element in $X^*(T)\otimes_\BZ \BR$ such that $\exp(2\pi \ii \lambda)= z^\vee$.
    Then there is a unique irreducible admissible $(\fgg_\BC, K_x)$-module $\pi$ with the following conditions:
\begin{enumerate}
\item[(a)] $\pi$ has infinitesimal character $\lambda$. (Here the infinitesimal character is represented by the $W$-orbit of $\lambda$.)
\item[(b)] $\pi|_{K_x}$ contains the irreducible $K_x$-representation with highest weight
\[
\Lambda =\lambda+\rho_{\lambda}-2\rho_{x,\lambda};
\]
\item[(c)] if $\Lambda'$ is the highest weight of a $K_x$-type in $\pi|_{K_x}$, then $\Lambda'$ is of the form
\[
\Lambda'=\Lambda+\sum_{\substack{\beta \in \Delta^+_\lambda \\ \beta(x) = -1}} n_\beta \beta \quad \text{for integers } n_\beta\ge 0.
\]
\end{enumerate}
This unique $(\fgg_\BC, K_x)$-module  $\pi$ is denoted by $A_{\fbb_\lambda, x}(\lambda)$. 
\end{thm}

\begin{defn}
	Let $\pi_x(\lambda)$ be the Casselman-Wallach representation of $G_{x}$ such that its Harish-Chandra module is given by $A_{\fbb_\lambda, x}(\lambda)$. 
\end{defn}

Under this notation, we have the following fact, see \cite[Section~5]{A} for example
\[
(x,\pi_x(\lambda)) \sim ({}^w x, \pi_{{}^w x}(w\cdot \lambda)) \text{ for all } w \in W.
\]

\subsection{Discrete Vogan packet}
\label{sec:discrete}
Let $B$ be the Borel subgroup of $\CG$ corresponding to $B^\vee$.
Let $\Delta^+$ be the set of roots in ${\rm Lie}\,B$; then the corresponding positive coroots $(\Delta^\vee)^+$ form the set of roots in ${\rm Lie}\,B^\vee $.

A discrete $L$-parameter $\phi: W_\BR \to G^\vee \rtimes \Gal(\BC/\BR)$, up to conjugation, may be assumed to be of the form:
\begin{equation} \label{HC}
\phi(z) = (z/\bar z)^\lambda, \quad z\in \BC^\times \subset W_\BR
\end{equation}
where $\lambda\in X^*(T)\otimes \BR$ satisfies 
\begin{itemize}
\item $\lambda$ is regular;
\item $\lambda$ is $B$-dominant;
\item $\exp(2\pi \ii \lambda)= z^\vee$, see \eqref{eq:z-vee} for the definition of $z^\vee$. 
\end{itemize}
In this case $\phi(j) = \exp(-\pi \ii \lambda)\rtimes c$, where as before $c\in \Gal(\BC/\BR)$ is the non-trivial element.
Moreover, the centralizer in $G^\vee$ of the image of $\phi$ is
\[
\Set{h\in T^\vee | h^2=1}.
\]
This set is disconnected and can be identified with the component group $\CS_\phi$ of $\phi$.  
Let $s$ be the map 
\[
 s : X^*(T)/ 2X^*(T) \rightarrow \Set{h\in T^\vee | h^2=1} = \CS_\phi 
\]
defined by $\gamma + 2X^*(T) \mapsto \exp(\pi \ii \gamma)$ for $\gamma \in X^*(T) = X_*(T^\vee)$.

For each $x\in \CX$, we define $\tau_x \in \widehat{\CS_\phi}$ by
\[
\tau_x(s(\gamma + 2X^*(T))) := \gamma(xx_b^{-1}), \quad \gamma \in X^*(T).
\]
Note that $(x x_b^{-1})^2 = 1$ for all $x\in \CX$.
Hence $\gamma(x x_b^{-1}) = 1$ for all $\gamma \in 2X^*(T)$, so $\tau_x$ is well-defined.
By the duality for algebraic tori, $\tau$ gives the following bijection 
\[
\tau : \CX \xrightarrow{\  \sim\  }\widehat{\CS_\phi},\qquad  x\longmapsto \tau_x.
\]

\begin{rmk}
	In Section~\ref{sec:ELP}, we use the local Weil group $W_\BR$ to define the $L$-group $^LG$ and the local $L$-parameter $\varphi$, while we use the local Galois group $\Gal(\BC/\BR)$ to define the local $L$-parameter $\phi$ here. They are 
    different in general. 
	We refer to \cite[Section~8]{GGP1} for the relation between $\phi$ and $\varphi$.
	In summary, the $G^\vee$-conjugacy class of $\phi$ is
	in bijection with the $G^\vee$-conjugacy class of $\varphi$, and there is a natural identification 
	of $\CS_\phi$ with $\CS_\varphi$; see \cite[Theorem~8~(iii)]{GGP1}. We also write $\Pi_\varphi[G^*]$ as $\Pi_\phi[G^*]$ under this identification. 
\end{rmk}

Recall that an irreducible Casselman-Wallach representation $\pi$ of $G^*$ is called large if the annihilator variety of $\pi$ is the whole nilpotent cone of $\frak g_\BC$ (\cite[Definition~6.1]{V78}).
In particular, such large representations are generic with respect to the Whittaker datum determined by their wavefront sets (\cite[Theorem~A]{Mat92}).

The following proposition gives a parameterization of the Vogan packet $\Pi_\phi[G^*]$. See \cite[Lemma 6.15]{A},
\cite[Appendix]{AA} and \cite[Section 4]{AK}. 

\begin{prop} \label{LLC}
	The following map   
\[
\widehat{\CS_\phi}\xrightarrow{\ \sim \ } \Pi_\phi[G^*],\qquad  \tau_x\longmapsto \pi_x(\lambda)
\]
is bijective. It takes the trivial character $\tau_{x_b}$ to a large discrete series $\pi_{x_b}(\lambda)$ of $G^*$. 
\end{prop}

\subsection{Associated variety} \label{ssec:AV}

Let $\theta$ be the Cartan involution of $G$, which induces the Cartan decomposition $\frak{g}_\BC= \frak{k}\oplus \frak{p}$. Let $K = G^\theta$ and $\CK = \CG^\theta$ be the corresponding subgroups of $G$ and $\CG$ respectively. For a finite length $(\frak{g}_\BC, K)$-module $X$, its associated variety ${\rm AV}(X)$ is a finite union of closures of nilpotent $K$-orbits on $\frak{p}$, which is an important invariant of 
$X$ (see e.g. \cite{AV21}). 
More precisely, given a $K$-stable good filtration on $X$, the associated graded  ${\rm gr}(X)$ is an $S(\frak g_\BC)$-module, where $S(\frak g_\BC)$ is the symmetric algebra of $\Fg_\BC$. By definition, the associated variety ${\rm AV}(X)$ of $X$ is the support of ${\rm gr}(X)$, viewed as a subvariety of $(\frak g_\BC/\frak k)^*\cong\frak p$. 
For an irreducible Casselman-Wallach representation $\pi$ of $G$, define the associated variety of $\pi$ to be:  $\AV(\pi):=\AV(\pi_{\rm HC})$, where $\pi_{\rm HC}$ is the Harish-Chandra module of $\pi$. 
 
Now we proceed to compute the associated variety of a discrete series representation.
Since the discrete series representation $\pi_x(\lambda)$ is realized as a cohomologically induced representation, we can use the following theorem to compute its associated variety. This theorem is well known to experts (see, for example, the introduction of \cite{Tr05}) but is difficult to find in a proper reference. The following very general version is from \cite[Proposition~8.2]{O15}.

\begin{thm}
	Let $\fqq$ be a $\theta$-stable parabolic subalgebra of $\fgg_\BC$ with Levi decomposition $\fqq = \fll \oplus \fuu$. Let $\overline{\fqq}$ denote the parabolic subalgebra opposite to $\fqq$ and $\overline{\fuu}$ be the nilpotent radical of $\overline{\fqq}$.
	For any one-dimensional character $\lambda$ of $\fll$, suppose that $X$ is an irreducible submodule of $A_\fqq(\lambda)$. Then the associated variety of $X$ is given by 
	\[
	\AV(X) = \CK (\fuu^*\cap \fpp^*) = \CK(\overline{\fuu}\cap \fpp).
	\] 
\end{thm}

\begin{rmk}
	It is important to point out that there is a subtle typo in \cite{Tr05}, which claims that the associated variety should be $\CK(\fuu\cap \fpp)$ instead of $\CK(\overline{\fuu}\cap \fpp)$. 
	Meanwhile, Trapa's combinatorial formula in fact computes $\CK(\overline{\fuu}\cap \fpp)$. 
    
	This can be seen from the example of $\mathrm{SU}(1,1)$. 
	Let $\fgg_\BC = \mathfrak{sl}(2,\BC)$ be the complexified Lie algebra and $x = \mathrm{diag}({\rm i},-{\rm i})$ be the diagonal matrix. Then $x^2 = -1 = (-1)^{2-1}$. 
	Let $\fqq$ be the upper triangular parabolic subalgebra of $\fgg_\BC$ with Levi decomposition $\fqq = \fll \oplus \fuu$, where $\fll$ is the subalgebra of diagonal matrices, so that $\overline{\fqq} = \fll \oplus \overline{\fuu}$ is the lower triangular parabolic subalgebra.
	Then $\fll$ is the Lie algebra of the maximal compact torus $T = \RU(1)$ of $\SU(1,1)$.
	The representations of $T$ are parameterized by $\BZ$.
	Let $\BC_\lambda$ denote the one-dimensional $U(\overline{\fqq})$-module on which $\fll$ acts by the character $\lambda$ and $\overline{\fuu}$ acts trivially.
	The Harish-Chandra module of a holomorphic discrete series representation is given by $\pi = U(\fgg_\BC)\otimes_{U(\overline{\fqq})} \BC_\lambda$.
	Note that $\pi$ has $T$-type given by $\lambda, \lambda+2, \lambda+4, \cdots$ (compare the definition of $A_\fqq(\lambda)$ in \cite{VZ}). 
	Let $\mathrm{gr}(\pi)$ denote the associated graded of $\pi$ with respect to the canonical filtration on $U(\fgg_\BC)$ induced by the tensor product. Then $\mathrm{gr}(\pi)$ has a natural grading, and as an $S(\fgg_\BC)$-module $\mathrm{gr}(\pi) \cong S(\fgg_\BC/\overline{\fqq}) \cong S(\fuu) = P(\fuu^*)$.  
	Here $S(\cdot)$ denotes the symmetric algebra and $P(\cdot)$ denotes the polynomial algebra on a vector space.
	Therefore, we see that $\AV(\pi) = \fuu^* \cong \overline{\fuu}$ under the identification of $\fgg_\BC^*$ with $\fgg_\BC$. 
\end{rmk}


\def\diag{\mathrm{diag}}
\def\GL{\mathrm{GL}}
\def\Mult{\mathrm{Mult}}
\def\Sign{\mathrm{Sign}}
\def\WF{\mathrm{WF}}

For classical groups, $\CG$ is viewed as a subgroup of $\GL(V)$, where $V$ is a complex vector space.
We always choose a basis of $V$ such that $T$ is a diagonal matrix and $B$ consists of upper triangular matrices.

In the following discussion, we let $\Sign(n)$ be the set of all sequences of signs of length $n$. An element $\delta \in \Sign(n)$ is written as
\[
\delta = (\delta_1, \cdots, \delta_n)
\] 
with $\delta_i \in \{\pm 1\}$.
We write $p_\delta = \#\set{i\in [n] | \delta_i = 1}, $ and $q_\delta = \#\set{i\in [n] | \delta_i = -1}$.

The group $\CS_\phi$ is identified with $\BZ_2^n$ with generators $e_i$.
Then $\widehat{\CS_\phi}$ is identified with $\Sign(n)$ by setting 
\[
\chi(e_i) := \chi_i \quad \text{for } i=1, \cdots, n
\] 
where $\chi = (\chi_1, \cdots, \chi_n) \in \Sign(n)$.
The bijection $\tau : \CX \xrightarrow{ \  \sim \  }\widehat{\CS_\phi}$ is now described by a map $\Sign(n)\to \Sign(n)$ sending $\delta $ to $\chi_\delta$ so that $\tau_{x_\delta} = \chi_\delta$. 

We make the following choices to explicate the various maps defined above and compute the wavefront set of $\pi_{x_b}(\lambda)$ using Theorem \ref{thm:TrapaYoung} below.
\begin{enumerate}
\item When $\CG = \GL_n(\BC)$, $\CX = \set{x \in T | x^2 = (-1)^{n-1}}$ and it is in bijection with $\Sign(n)$ via the map  
\[
\delta \mapsto x_\delta = (-{\rm i})^{n-1}\diag(\delta_1, \cdots, \delta_n).\]
Let $(p_x, q_x) = (p_\delta, q_\delta)$.
The real form $G_{x}$ is identified with $\RU(p_x, q_x)$.
In this case, 
\[
x_b = {\rm i}^{n-1}\diag(1,-1,\cdots, (-1)^{n-1}) 
\]
 and $G^* = G_{x_b} =\RU(\lceil n/2\rceil, \lfloor n/2\rfloor)$. 
The bijection $\tau: \CX \xrightarrow{ \  \sim \  }\widehat{\CS_\phi}$ is given by 
 \[
 (\chi_\delta)_i = (-1)^{n-i}\delta_i \text{ for } i=1, \cdots, n.
 \]
In this case, we obtain 
 \[
 \WF(\pi_{x_b}(\lambda)) = 
    \begin{ytableau}
    \pm & \mp & \pm & \cdots &+  
    \end{ytableau}.
 \]

\item When $\CG = \SO_{2n}(\BC)$, 
$\CX = \set{x \in T | x^2 = 1}$ and it is in bijection with $\Sign(n)$ via the map  
\[
\delta \mapsto x_\delta = \diag(\delta_1, \cdots, \delta_n, \delta_n, \cdots, \delta_1).
\]
Let $(p_x, q_x) = (2 p_\delta, 2 q_\delta)$.
The real form $G_{x}$ is identified with $\SO(p_x, q_x)$.
In this case, 
\[
x_b = \diag((-1)^{n-1},\cdots,-1, 1, 1, -1, \cdots,(-1)^{n-1})
\]
and $G^* = G_{x_b} = \SO(n,n)$.
The bijection $\tau: \CX \xrightarrow{ \  \sim \  }\widehat{\CS_\phi}$ is given by 
 \[
 (\chi_\delta)_i = (-1)^{n-i}\delta_i \text{ for } i=1, \cdots, n.
 \]
 In this case, we obtain
 \[
\WF(\pi_{x_b}(\lambda)) =
\begin{ytableau}
- &+ & \cdots &+&- \\ 
+
\end{ytableau}
.
 \]

\item When $\CG = \SO_{2n+1}(\BC)$, $\CX = \set{x \in T | x^2 = 1}$. 
Then $\CX$ is in bijection with $\Sign(n)$ via the map  
\[
\delta \mapsto x_\delta = \diag( \delta_1,  \cdots, \delta_n, 1,  \delta_n, \cdots, \delta_1).
\]
Recall that $\delta_0 = (-1)^n {\rm disc}(V)$ determines the pure inner class of $V$. Let
\[
(p_x, q_x) = 
\begin{cases}
(2 p_\delta +1 , 2 q_\delta ) & \text{if } \delta_0 = 1. \\	
(2 p_\delta , 2 q_\delta +1 ) & \text{if } \delta_0 = -1.
\end{cases}
\] 
The real form $G_{x}$ is identified with $\SO(p_x, q_x)$. 
In this case, 
\[
x_b  = \diag((-1)^n, \cdots, -1,  1,  -1, \cdots,(-1)^n).
\]
When $\delta_0=1$, \[
G^* = \begin{cases}
\SO(n+1,n) & \text{if $n$ is even;} \\
\SO(n,n+1) & \text{if $n$ is odd.} \\
\end{cases}
\]
When $\delta_0=-1$, 
\[
G^* = \begin{cases}
\SO(n,n+1) & \text{if $n$ is even;} \\
\SO(n+1,n) & \text{if $n$ is odd.} \\
\end{cases}
\]
The bijection $\tau: \CX \xrightarrow{ \  \sim \  }\widehat{\CS_\phi}$ is given by 
 \[
 (\chi_\delta)_i = (-1)^{n-i} {\rm disc}(V)  \delta_i = (-1)^i \delta_0\delta_i \text{ for } i=1, \cdots, n.
 \]
In this case, we obtain 
 \[
 \WF(\pi_{x_b}(\lambda)) =
 \begin{cases}
 \begin{ytableau}
 \pm &\mp & \cdots & + &\cdots &\mp & \pm \\ 
 \end{ytableau}
 & \text{if $\delta_0 = 1$, and $n$ is even (resp. odd);} \\
 \begin{ytableau}
 \mp &\pm & \cdots & - &\cdots &\pm & \mp \\ 
 \end{ytableau}
 & \text{if $\delta_0 = -1$, and $n$ is even (resp. odd).} \\
 \end{cases}
 \]

\item When $\CG = \Sp_{2n}(\BC)$, 
\[
\CX = \set{x \in T | x^2 = -1}.
\]
It is in bijection with $\Sign(n)$ via the map  
\[
\delta \mapsto x_\delta = {\rm i}^{-1} \diag(\delta_1, \cdots, \delta_n, \delta_n, \cdots, \delta_1).
\]
Let
\[
(p_x, q_x) = (2 p_\delta, 2 q_\delta).
\]
Then
$p_x = q_x =n$ and the real form $G_{x}$ is identified with $\Sp_{2n}(\BR)$.  
In this case, 
\[
x_b = {\rm i}^{2n-1}\diag(1,-1,\cdots, (-1)^{n-1}, (-1)^{n-1}, \cdots, -1, 1)
\]
and $G^* = \Sp_{2n}(\BR)$.  
The bijection $\tau: \CX \xrightarrow{ \  \sim \  }\widehat{\CS_\phi}$ is given by 
\[
(\chi_\delta)_i = (-1)^{n - i + 1} \delta_i \text{ for } i=1, \cdots, n.
\]	
In this case, 
 \[
 \WF(\pi_{x_b}(\lambda)) =
 \begin{ytableau}
  + & - & \cdots & + & \cdots & + & - \\ 
 \end{ytableau}.
 \]
\end{enumerate}

\medskip

Following \cite{Tr05}, the Levi part of ${\frak b}_{x_\delta}$ is given by  
\[
{\frak u}(p_1, q_1)\times  {\frak u}(p_2,q_2)\times \cdots \times {\frak u}(p_n, q_n),
\]
where 
\[
(p_i, q_i)= \begin{cases}
(1,0) & \text{if } \delta_i = 1, \\
(0,1) & \text{if } \delta_i = -1.
\end{cases}
\]
Write $\CO(\delta)$ for the $\CK$-orbit such that 
\[
\AV(\pi_{x_\delta}) = \overline{\CO(\delta)}.
\]


The signed Young tableau that parametrizes $\CO(\delta)$ is given case by case in \cite[Sections 3, 4, 7]{Tr05} for $G^* = \RU_n$, $\Sp_{2n}(\BR)$, and $\SO_{\frak n}$ respectively. The result also works for  
$G = \Mp_{2n}(\BR)$ as the computation of $\CK \cdot(\overline{\fuu}\cap \fpp)$ has nothing to do with the covering. 

Denote by $c_{r, s}$ a column with $r$ pluses and $s$ minuses. For a signed Young tableau $T$, 
let $T\oplus c_{r, s}$ be the equivalence class of the signed Young tableau formed by placing the entries of $c_{r, s}$ at the row-ends of  $T$ such that
\begin{itemize}
\item at most one sign is added to each row-end;
\item each sign is added to as high a row as possible.
\end{itemize}
One  can similarly define $c_{r, s} \oplus T$,   but the entries of $c_{r, s}$ are added to the row-starts of $T$.  It can be shown that the operation 
$\oplus$ is well-defined and associative in a suitable sense (see \cite[Section 2]{Tr05} for detailed explanations). 
Denote by $c_{\delta_i}$ the singleton Young tableau with  the signature $(1,0)$ if $\delta_i = 1$ and the signature $(0,1)$ if $\delta_i = -1$. 

\begin{thm}\label{thm:TrapaYoung}
	Let the notations be as above. Then the signed Young tableau parametrizing $\CO(\delta)$ is given by the following.

(i) $G^*=\RU_n$.  Then 
\[
\CO(\delta) = c_{\delta_1} \oplus c_{\delta_2}\oplus \cdots \oplus c_{\delta_n}. 
\]

(ii) $G^* = \SO_{2n}$. Then 
\[
\CO(\delta) = c_{\delta_n} \oplus \cdots \oplus c_{\delta_2}  \oplus c_{\delta_1} \oplus c_{\delta_1} \oplus c_{\delta_2}\oplus \cdots \oplus c_{\delta_n}. 
\]

(iii) $G^*= \SO_{2n+1}$. Then 
\[
\CO(\delta) = c_{\delta_n} \oplus \cdots \oplus c_{\delta_2}  \oplus c_{\delta_1} \oplus c_{\delta_0}\oplus c_{\delta_1} \oplus c_{\delta_2}\oplus \cdots \oplus c_{\delta_n}. 
\]

(iv) $G=\Sp_{2n}(\BR)$. Then
\[
\CO(\delta) = c_{\delta_n} \oplus \cdots \oplus c_{\delta_2}  \oplus c_{\delta_1} \oplus c_{-\delta_1} \oplus c_{-\delta_2}\oplus \cdots \oplus c_{-\delta_n}.
\] 

(v) $G = \Mp_{2n}(\BR)$. Then 
\[
\CO(\delta) = c_{\delta_n} \oplus \cdots \oplus c_{\delta_2}  \oplus c_{\delta_1} \oplus c_{-\delta_1} \oplus c_{-\delta_2}\oplus \cdots \oplus c_{-\delta_n}.
\]
\end{thm}

\subsection{Examples}

To illustrate the algorithm above, we give examples for unitary and orthogonal groups.
More examples can be executed on the website \cite{Luo}.

\subsubsection{Unitary groups} 

We first carry out the algorithm for a discrete Harish-Chandra parameter $(\lambda, \delta)$ for $\RU(10,8)$.
Write
\[
\lambda=\lambda_{\alpha_1} \oplus \cdots \oplus \lambda_{\alpha_{18}} 
\]
where $\alpha_1<\alpha_2<\cdots<\alpha_{18}$ are half of odd integers, and $\delta$ is
\[
+++++-----++++--+-.
\]
Then by the associativity of the operation $\oplus$, one has
\[
\begin{split}
\CO(\delta) &= c_{\delta_1} \oplus c_{\delta_2}\oplus \cdots \oplus c_{\delta_{18}} = c_{5, 0} \oplus c_{0, 5} \oplus c_{4, 0} \oplus c_{0, 2} \oplus c_{1, 0} \oplus c_{0, 1} \\
& = 
\begin{ytableau} 
+ & - & + & - & + & - \\
+ & - & + & - \\
+ & - & + \\
+ & - & + \\
+ & - \\
\end{ytableau}
\end{split}
\]

\subsubsection{Special orthogonal groups}
Since the formula for special orthogonal groups involves an auxiliary sign, we include one example to clarify this feature.
Consider a discrete Harish-Chandra parameter $(\lambda,\delta)$ for $\SO(16,15)$ with character
\[
--+++++----++-+.	
\]
By the same argument as in Subsection~\ref{ssec:AW-SO}, $\delta_0 = -1$.
Then
\[
\begin{aligned}
\CO(\delta) &= c_{\delta_{15}} \oplus \cdots \oplus c_{\delta_2} \oplus c_{\delta_1} \oplus c_{\delta_0}\oplus c_{\delta_1} \oplus c_{\delta_2}\oplus \cdots \oplus c_{\delta_{15}}
    \\
    &= c_{1, 0} \oplus c_{0, 1} \oplus c_{2, 0} \oplus c_{0, 4} \oplus c_{5, 0} \oplus c_{0, 2} \oplus c_{0, 1} \oplus c_{0, 2} \oplus c_{5, 0} \oplus c_{0, 4} \oplus c_{2, 0} \oplus c_{0, 1} \oplus c_{1, 0}
    \\
    &=
\begin{ytableau} 
+&- & + & - & + & - & + & - & + & - & + \\
+&- & + & - & + & - & + \\
- & + & - & + & - \\
- & + & - & + & - \\
+ & - & +\\
\end{ytableau}
\end{aligned}
\]

 \section{Arithmetic Wavefront Set and Associated Variety} \label{sec:AWFAV}

Following Theorem \ref{FRS-WFS} (\cite[Theorem 9.2]{JLZ25}),  when the Vogan packet $\Pi_\varphi[G^*]$ is not a singleton, associated to each $\pi =\pi(\varphi,\chi)\in \Pi_\varphi[G^*]$,  there is a real distinguished nilpotent orbit 
\[
\CO(\pi) = \CO(\varphi,\chi) \subset {\frak g}, 
\]
which is the unique maximal element of the arithmetic wavefront set ${\rm WF_{ari}}(\pi)$ defined in {\it loc. cit.}  under the real topological order. The case of discrete series representations has been recalled in Section \ref{ssec:AWF}. 

In terms of the Kostant-Sekiguchi correspondence (see \cite[Section 2]{AV21} for instance) given by 
\begin{eqnarray}\label{KSC}
 \Set{\text{nilpotent $G$-orbits on $\frak g$}}  & \stackrel{\sim}{\longrightarrow} & \Set{\text{nilpotent $\CK$-orbits on $\frak p$}}, \\
   \CO & \longmapsto & \CO_{\frak p}, \nonumber
 \end{eqnarray}
we prove the following result.

\begin{thm} \label{thm:AV}
Assume that $\pi = \pi(\varphi,\chi)$ is an irreducible Casselman-Wallach representation of $G$ with a generic $L$-parameter $\varphi$ and the Vogan packet $\Pi_\varphi[G^*]$ being not a singleton. Then the associated variety ${\rm AV}(\pi)$ is the closure of 
\[
\CO(\pi)_{\frak p} = \CO(\varphi,\chi)_{\frak p},
\]
the image of $\CO(\pi) = \CO(\varphi,\chi)$ under the Kostant-Sekiguchi correspondence.
\end{thm}

By the relation between the associated cycle and the wavefront cycle established in \cite{SV00},  we have the following immediate consequence of  Theorem \ref{thm:AV}. 
Note that if $\Pi_\varphi[G^*]$ is a singleton, then $G$ is quasi-split and 
 $\pi$ is generic with respect to every Whittaker datum of $G$, in which case ${\rm AV}(\pi)$ is the whole nilpotent cone in $\frak{p}$. 

\begin{cor}\label{AWF=AV}
Assume that $\pi = \pi(\varphi,\chi)$ is an irreducible Casselman-Wallach representation of $G$ with a generic $L$-parameter $\varphi$. Then Conjecture \ref{conj:wfss} holds for both ${\rm WF_{tr}}(\pi)$ and ${\rm WF_{\rm ari}}(\pi)$:
\[
 {\rm WF_{tr}}(\pi)^{r-\max} = {\rm WF_{\rm ari}}(\pi)^{r-\max}.
\]
\end{cor}

\subsection{Reduction} 
We first reduce the proof of Theorem \ref{thm:AV} to the case where $\pi$ is a discrete series representation.
Recall that $\varphi$ is given by \eqref{varphi}, \eqref{gp} and \eqref{tau}. First assume that $G$ is linear. Since $\varphi$ is generic, it follows from 
\cite{SV80} that $\pi(\varphi,\chi)$ is an irreducible normalized parabolic induction
\be \label{ind}
\pi(\varphi,\chi) \cong |\cdot|_E^{s_1}\pi_{\tau_1} \times \cdots \times |\cdot|_E^{s_l}\pi_{\tau_l} \times \pi(\varphi_{\rm gp},\chi),
\ee
where $\pi_{\tau_j}$ is the (limit of) discrete series representation of ${\rm GL}_1(E)$ or ${\rm GL}_2(E)$ with $L$-parameter $\tau_j$, and $\pi(\varphi_{\rm gp}, \chi)$ is a representation of a classical group $G'$ in the Witt tower of $G$, with the $L$-parameter $\varphi_{\rm gp}$ as in \eqref{gp}. A slight modification of \eqref{ind} also holds when $G=\Mp_{2n}(\BR)$. 

Consider a discrete $L$-parameter $\varphi'$ of $G'$ of the form 
\[
\varphi' = \rho_{\alpha_1'}\oplus \cdots \oplus \rho_{\alpha_m'} \oplus \varphi_{\rm quad}, 
\]
where $m= m_1+\cdots+m_r$ and $\alpha'_1<\cdots < \alpha'_m$, with $\varphi_{\rm quad}$ as in \eqref{gp}. There is a natural embedding of centralizers 
\[
C_{G'^\vee}({\rm im}\, \varphi')   \hookrightarrow C_{G'^\vee}({\rm im}\, \varphi_{\rm gp}), 
\]
which in turn induces a surjective homomorphism of component groups
\[
p: \CS_{\varphi'} \twoheadrightarrow \CS_{\varphi_{\rm gp}}. 
\]
Let $\chi ' = \chi\circ p \in \widehat{\CS_{\varphi'}}$. Then $\pi(\varphi_{\rm gp},\chi)$ and $\pi(\varphi',\chi')$ lie in the same coherent family (see \cite[Chapter 7]{V81}), and by the translation principle (\cite[Lemma 2.7]{V79}) it holds that
\[
\AV(\pi(\varphi_{\rm gp},\chi))= \AV(\pi(\varphi',\chi')).
\]

Assume that Theorem \ref{thm:AV} holds for the discrete series representation $\pi(\varphi',\chi')$, so that  $\AV(\pi(\varphi',\chi'))$ is the closure of 
$\CO(\varphi',\chi')_{\frak p}$. By the algorithm in \cite{JLZ25}, we have
$\CO(\varphi_{\rm gp},\chi) = \CO(\varphi',\chi')$, hence $\AV(\pi(\varphi_{\rm gp},\chi))$ is the closure of $\CO(\varphi_{\rm gp},\chi)_{\frak p}$. Write 
\[
\CO(\varphi_{\rm gp},\chi) = ((p_1,\epsilon_1), (p_2,\epsilon_2), \ldots, (p_k, \epsilon_k)). 
\]
Put $t =\dim \tau$. The  algorithm also gives that
\[
\CO(\varphi, \chi) = ((p_1+2t,\epsilon_1), (p_2,\epsilon_2), \ldots, (p_k, \epsilon_k)).
\]

By \cite[Proposition 7.13]{JLZ25}, $\CO(\varphi_{\rm gp},\chi)$ is distinguished. In other words, it holds that if $p_i = p_j$ then $\epsilon_i = \epsilon_j$. 
It follows that the Lusztig-Spaltenstein induction (\cite{LS79}) of the regular nilpotent orbit in $\frak{gl}_t(E)$ and the nilpotent orbit $\CO(\varphi_{\rm gp},\chi)$ in $\frak{g}'$ is the nilpotent orbit  $\CO(\varphi, \chi)$ in ${\frak g}$. Then by \eqref{ind} and \cite{BB},  $\AV(\pi(\varphi, \chi))$ is the closure of $\CO(\varphi,\chi)_{\frak p}$. This proves that Theorem \ref{thm:AV} holds for 
$\pi(\varphi, \chi)$.

\subsection{Proof for discrete series} \label{ssec:DS}

The rest of this section is devoted to the proof of Theorem \ref{thm:AV} for discrete series representations. Let $\pi = \pi(\varphi, \chi)$ be a discrete series representation of $G$. We need to show that $\AV(\pi)$ is the closure of $\CO(\varphi,\chi)_{\frak p}$. 
	
The proof is case-by-case. We first outline the common strategy.
	\begin{enumerate}
		\item Translate the algorithm in Section \ref{ssec:AV} into a recursive algorithm parallel to the one in 
		Section \ref{ssec:AWF}.
            Concretely, this produces the first row of the signed Young tableau $\CO(\delta)$ together with a new Harish-Chandra parameter $\delta'$, which then generates the remaining rows of $\CO(\delta)$. We call $\delta'$ an {\bf HC-descent} of $\delta$.
		\item Show that the first rows of $\CO(\delta)$ and $\CO(\varphi,\chi)$ produced by the two algorithms agree.
		\item Show that the HC-descent $\delta'$ of $\delta$ and the $L$-descent $\chi'$ of $\chi$ match as in Section \ref{ssec:AV}.
	\end{enumerate}
	Then by induction,  the signed Young tableaux $\CO(\delta)$ and $\CO(\varphi,\chi)$ generated by both algorithms are the same. 
    
    We present the details for each family of real classical groups.

\medskip
\noindent
{\bf The unitary groups:}

\medskip
\noindent
(i) {\it Recursive algorithm}. 
			We first divide the sequence $\delta$ into consecutive subsequences, where the signs in each subsequence are the same:
			\[
				\delta = (m_1  \delta_{k_1}, m_2 \delta_{k_2}, \dots, m_{\ell} \delta_{k_\ell}).
			\]
			For example, if \(\delta = (+1, +1, -1, -1, -1, +1)\), then $(m_{1}, m_{2}, m_{3}) = (2, 3, 1)$  and 
			$(\delta_{k_1}, \delta_{k_2}, \delta_{k_3}) = (\delta_1, \delta_3, \delta_6) = (+1, -1, +1)$. 
			By the associativity of the operation $\oplus$, we can rewrite $\CO(\delta)$ as 
						\[
			\CO(\delta)  = c_{m_1 \delta_{k_1}} \oplus \cdots \oplus c_{m_\ell \delta_{k_\ell}},
			\]
			where $c_{m_i \delta_{k_i}}$ denotes the column with  $m_i$ signs $\delta_{k_i}$. 
			Then the first row of  $\CO(\delta)$ is given by picking one sign from each $m_i \delta_{k_i}$, and the remaining rows of $\CO(\delta)$ are generated by the same algorithm for the HC-descent 
			\be\label{HCD}
			\delta': = ((m_1-1)\delta_{k_1}, (m_2-1)\delta_{k_2}, \ldots, (m_\ell-1)\delta_{k_\ell}). 
			\ee
			This indeed gives a recursive algorithm for the associated variety.

\medskip
\noindent
(ii) {\it Matching the first rows.}
		 We now show that the first rows of $\CO(\delta)$ and $\CO(\varphi,\chi)$ are the same. It suffices to check that their lengths and last signs are respectively the same.  			Recall from Section \ref{ssec:AV} that $\delta$ and $\chi$ are related by 
		 \[
		 \delta_i = (-1)^{n-i} \chi_i,\quad i=1,2,\dots, n.
		 \]
It follows that
		\be \label{sgnHC}
		\sgn(\chi) = \set{ 1\leq i< n |  \delta_i \delta_{i+1}=1}. 
		\ee		
	Hence the length of the first row of  $\CO(\delta)$ is 
			\[
				\ell = 1 +  \# \sgn(\delta)  = 1 + (n - 1) -  \# \sgn(\chi) = n- s_\chi  = p_1.
			\]
			The last sign of the first row of $\CO(\delta)$ is clearly $\delta_n = \chi_n = \epsilon_1$. This proves that 
			the first rows of  $\CO(\delta)$ and $\CO(\varphi, \chi)$ are the same.

\medskip
\noindent
(iii) {\it Matching HC-descent and $L$-descent.} Finally, we show that $\delta'$ and $\chi'$ are related as in Section \ref{ssec:AV}. 
	From \eqref{HCD} and \eqref{sgnHC} it is clear that 
	\[
	\delta' = (\delta_j : j\in {\rm sgn}(\chi)). 
	\]
Recall from Section \ref{ssec:AWF} that 
			\[
			\chi' = ((-1)^{n-j+1} \chi_n : j\in {\rm sgn}(\chi)). 
			\]
			Take $j\in {\rm sgn}(\chi)$, and suppose that it is the $k$-th largest element in  $\sgn(\chi)$.
			It suffices to verify that
			\begin{equation}\label{HCLD}
							\delta_j  =  (-1)^{k - 1} \cdot (-1)^{n-j+1}\chi_n  = (-1)^{n-j-k}\chi_n.
			\end{equation}
			In fact, one has
			\[
				n - j - k = \# \set{ j \leq i < n | i \notin \sgn(\chi)} = \# \set{ j \leq i < n | i \in  \sgn(\delta)}.
			\]
			Therefore,  $\delta_j$ and $\delta_n = \chi_n$ differ by $(-1)^{n-j-k}$, which shows that 
			\eqref{HCLD} holds.

\medskip
\noindent
{\bf The even special orthogonal groups:}

\medskip
\noindent
(i) {\it Recursive algorithm.}
			We first write $\delta = (m_1  \delta_{k_1}, m_2 \delta_{k_2}, \dots, m_{\ell} \delta_{k_\ell})$ as in the unitary case.
			By the associativity of the operation $\oplus$, we can rewrite $\CO(\delta)$ as 
						\[
			\CO(\delta)  = c_{m_\ell \delta_{k_\ell}} \oplus \cdots \oplus c_{m_2 \delta_{k_2}} \oplus c_{2m_1 \delta_{k_1}} \oplus c_{m_2 \delta_{k_2}} \oplus \cdots \oplus c_{m_\ell \delta_{k_\ell}},
			\]
			Then the first row of  $\CO(\delta)$ is given by picking one sign from each $m_i \delta_{k_i}, i \geq 2$ and $2m_1 \delta_{k_1}$, and the remaining rows of $\CO(\delta)$ are generated by the same algorithm for the HC-descent 
			\be\label{HCDSo}
			\delta': = ((m_1-1)\delta_{k_1}, (m_2-1)\delta_{k_2}, \ldots, (m_\ell-1)\delta_{k_\ell}), \quad \delta_0' = \delta_1.
			\ee
			This indeed gives a recursive algorithm for the associated variety.
		
\medskip
\noindent		
		(ii) {\it Matching the first rows}. 
		 We now show that the first rows of $\CO(\delta)$ and $\CO(\varphi,\chi)$ are the same. It suffices to check that their lengths and last signs are respectively the same.  			
         Recall from Section \ref{ssec:AV} that $\delta$ and $\chi$ are related by 
		 \[
		 \delta_i = (-1)^{n-i} \chi_i,\quad i=1,2,\dots, n.
		 \]
It follows that
		\be \label{sgnHCSo}
		\sgn(\chi) = \set{ 1\leq i< n |  \delta_i \delta_{i+1}=1}. 
		\ee		
	Hence the length of the first row of  $\CO(\delta)$ is 
			\[
				2\ell - 1 = 1 + 2 \# \sgn(\delta)  = 1 + 2 (n - 1 -  \# \sgn(\chi)) = 2(n- s_\chi) - 1  = p_1.
			\]
			The last sign of the first row of $\CO(\delta)$ is $\delta_n = \chi_n = \epsilon_1$. This proves that 
			the first rows of  $\CO(\delta)$ and $\CO(\varphi, \chi)$ are the same. 

\medskip
\noindent
	(iii)  {\it Matching HC-descent and $L$-descent}.  Finally, we show that the auxiliary signs are also matched, i.e., $\delta_0' = (-1)^{n - s_\chi - 1}\chi_n$, and $\delta'$ and $\chi'$ are related as in Section \ref{ssec:AV}. 

    For the auxiliary sign, note that $n - s_\chi - 1 = s_\delta$, so $\delta_0' = (-1)^{n - s_\chi - 1}\chi_n$.
    
	From \eqref{HCDSo} and \eqref{sgnHCSo} it is clear that 
	\[
	\delta' = (\delta_j : j\in {\rm sgn}(\chi)). 
	\]
Recall from Section \ref{ssec:AWF} that 
			\[
			\chi' = ((-1)^{j}  : j\in {\rm sgn}(\chi)). 
			\]
			Take $j\in {\rm sgn}(\chi)$, and suppose that it is the $k$-th smallest element in  $\sgn(\chi)$.
			It suffices to verify that
			\begin{equation}\label{HCLDSo}
							\delta_j  =  (-1)^{k} \cdot \delta_0' \cdot (-1)^{j}  = (-1)^{j-k} \delta_{1}.
			\end{equation}
			In fact, one has
			\[
				j - k = \# \set{ 1 \leq i < j | i \notin \sgn(\chi)} = \# \set{ 1 \leq i < j | i \in  \sgn(\delta)}.
			\]
			Therefore,  $\delta_j$ and $\delta_1$ differ by $(-1)^{j-k}$, which shows that 
			\eqref{HCLDSo} holds.

\medskip
\noindent
{\bf The odd special orthogonal groups:}

\medskip
\noindent
(i) {\it Recursive algorithm}. 
			We first write $\delta = (m_1  \delta_{k_1}, m_2 \delta_{k_2}, \dots, m_{\ell} \delta_{k_\ell})$ as in the unitary case.
			By the associativity of the operation $\oplus$, we can rewrite $\CO(\delta)$ as 
						\[
			\CO(\delta)  = c_{m_\ell \delta_{k_\ell}} \oplus \cdots \oplus c_{m_1 \delta_{k_1}} \oplus c_{\delta_0} \oplus c_{m_1 \delta_{k_1}} \oplus \cdots \oplus c_{m_\ell \delta_{k_\ell}},
			\]
			Then:
            \begin{itemize}
                \item if $\delta_0\delta_1 = -1$, then the first row of  $\CO(\delta)$ is given by picking one sign from each $m_i \delta_{k_i}, i \geq 1$ and $\delta_0$, and the remaining rows of $\CO(\delta)$ are generated by the same algorithm for the HC-descent 
			\be\label{HCDSoOdd1}
			\delta': = ((m_1-1)\delta_{k_1}, (m_2-1)\delta_{k_2}, \ldots, (m_\ell-1)\delta_{k_\ell}).
			\ee
                \item if $\delta_0\delta_1 = 1$, then the first row of  $\CO(\delta)$ is given by picking one sign from each $m_i \delta_{k_i}, i \geq 2$ and one $\delta_0 = \delta_1$ from $(2m_1 + 1)\delta_1$, and the remaining rows of $\CO(\delta)$ are generated by the same algorithm for the HC-descent 
			\be\label{HCDSoOdd2}
			\delta': = (m_1\delta_{k_1}, (m_2-1)\delta_{k_2}, \ldots, (m_\ell-1)\delta_{k_\ell}).
			\ee
            \end{itemize}
			These indeed give a recursive algorithm for the associated variety.
		
\medskip
\noindent		
		(ii) {\it Matching the first rows}. 
		 We now show that the first rows of $\CO(\delta)$ and $\CO(\varphi,\chi)$ are the same. It suffices to check that their lengths and last signs are respectively the same.  			
         Recall from Section \ref{ssec:AV} that $\delta$ and $\chi$ are related by 
		 \[
		 \delta_i = (-1)^{i} \delta_0 \chi_i,\quad i=1,2,\dots, n.
		 \]
It follows that
		\be \label{sgnHCSoOdd}
		\sgn(\chi) = \set{ 0\leq i< n |  \delta_i \delta_{i+1}=1}. 
		\ee		
	Hence the length of the first row of  $\CO(\delta)$ is 
			\[
				2\ell + 1 = 1 + 2 \# \sgn(\delta)  = 1 + 2 (n -  \# \sgn(\chi)) = 2(n- s_\chi) + 1  = p_1.
			\]
			The last sign of the first row of $\CO(\delta)$ is $\delta_n = (-1)^n \delta_0 \chi_n = \epsilon_1$. This proves that 
			the first rows of  $\CO(\delta)$ and $\CO(\varphi, \chi)$ are the same. 

        \medskip
\noindent
	(iii)  {\it Matching HC-descent and $L$-descent}.  Finally, we show that $\delta'$ and $\chi'$ are related as in Section \ref{ssec:AV}. 

	From \eqref{HCDSoOdd1}, \eqref{HCDSoOdd2} and \eqref{sgnHCSoOdd} it is clear that 
	\[
	\delta' = (\delta_j : j\in {\rm sgn}(\chi)). 
	\]
Recall from Section \ref{ssec:AWF} that 
			\[
			\chi' = ((-1)^{j - 1} \delta_0 \chi_n : j\in {\rm sgn}(\chi)). 
			\]
			Take $j\in {\rm sgn}(\chi)$, and suppose that it is the $k$-th largest element in  $\sgn(\chi)$.
			It suffices to verify that
			\begin{equation}\label{HCLDSoOdd}
							\delta_j  =  (-1)^{k - 1} \cdot (-1)^{j - 1} \delta_0 \chi_n  = (-1)^{n-j-k} \delta_{n}.
			\end{equation}
			In fact, one has
			\[
				n - j - k = \# \set{ j \leq i < n | i \notin \sgn(\chi)} = \# \set{ j \leq i < n | i \in  \sgn(\delta)}.
			\]
			Therefore,  $\delta_j$ and $\delta_n$ differ by $(-1)^{n-j-k}$, which shows that 
			\eqref{HCLDSoOdd} holds.

  \medskip
\noindent
 {\bf The symplectic/metaplectic groups:}

\medskip
\noindent
(i) {\it Recursive algorithm}. 
			We first write $\delta = (m_1  \delta_{k_1}, m_2 \delta_{k_2}, \dots, m_{\ell} \delta_{k_\ell})$ as in the unitary case.
			By the associativity of the operation $\oplus$, we can rewrite $\CO(\delta)$ as 
						\[
 			\CO(\delta)  = c_{m_\ell \delta_{k_\ell}} \oplus \cdots \oplus c_{m_1 \delta_{k_1}} \oplus c_{-m_1 \delta_{k_1}} \oplus \cdots \oplus c_{-m_\ell \delta_{k_\ell}},
			\]
 			Then the first row of  $\CO(\delta)$ is given by picking one sign from each $m_i \delta_{k_i}$ and $-m_i \delta_{k_i}$, and the remaining rows of $\CO(\delta)$ are generated by the same algorithm for the HC-descent 
			\be\label{HCDSp}
			\delta': = ((m_1-1)\delta_{k_1}, (m_2-1)\delta_{k_2}, \ldots, (m_\ell-1)\delta_{k_\ell}). 
 			\ee
 			This indeed gives a recursive algorithm for the associated variety.
		
	\medskip
\noindent	
		(ii) {\it Matching the first rows}. 
		 We now show that the first rows of $\CO(\delta)$ and $\CO(\varphi,\chi)$ are the same. It suffices to check that their lengths and last signs are respectively the same.  			
         Recall from Section \ref{ssec:AV} that $\delta$ and $\chi$ are related by 
		 \[
		 \delta_i = (-1)^{n-i+1} \chi_i,\quad i=1,2,\dots, n.
 		 \]
 It follows that
		\be \label{sgnHCSp}
		\sgn(\chi) = \set{ 1\leq i< n |  \delta_i \delta_{i+1}=1}. 
 		\ee		
 	Hence the length of the first row of  $\CO(\delta)$ is 
			\[
				2\ell = 2 + 2 \# \sgn(\delta)  = 2 + 2 (n - 1 -  \# \sgn(\chi)) = 2(n- s_\chi)  = p_1.
 			\]
 			The last sign of the first row of $\CO(\delta)$ is $-\delta_n = \chi_n = \epsilon_1$. This proves that 
 			the first rows of  $\CO(\delta)$ and $\CO(\varphi, \chi)$ are the same. 

\medskip
\noindent		
 	(iii)  {\it Matching HC-descent and $L$-descent}.  Finally, we show that $\delta'$ and $\chi'$ are related as in Section \ref{ssec:AV}. 
 	From \eqref{HCDSp} and \eqref{sgnHCSp} it is clear that 
 	\[
 	\delta' = (\delta_j : j\in {\rm sgn}(\chi)). 
 	\]
 Recall from Section \ref{ssec:AWF} that 
 			\[
			\chi' = ((-1)^{n-j+1} \chi_n : j\in {\rm sgn}(\chi)). 
			\]
			Take $j\in {\rm sgn}(\chi)$, and suppose that it is the $k$-th largest element in  $\sgn(\chi)$.
			It suffices to verify that
			\begin{equation}\label{HCLDSp}
 							\delta_j  =  (-1)^{k} \cdot (-1)^{n-j+1}\chi_n  = (-1)^{n-j-k + 1}\chi_n.
 			\end{equation}
 			In fact, one has
 			\[
				n - j - k = \# \set{ j \leq i < n | i \notin \sgn(\chi)} = \# \set{ j \leq i < n | i \in  \sgn(\delta)}.
			\] 			
 Therefore,  $\delta_j$ and $\delta_n = -\chi_n$ differ by $(-1)^{n-j-k}$, which shows that 
			\eqref{HCLDSp} holds.
\subsection{Examples}

To illustrate the proof in Subsection~\ref{ssec:DS}, we give examples for unitary and orthogonal groups.

\subsubsection{Unitary groups}

We first carry out the algorithm for a discrete series representation $\pi = \pi(\varphi, \chi)$ for $\RU(10,8)$.
Let the character $\chi$ be
\[
-+-+--+-+--+-++---,
\]
Then by Section \ref{ssec:AV}, the corresponding character $\delta$ of Harish-Chandra parameter is
\[
+++++-----++++--+-.
\]
We can rewrite $\CO(\delta)$ as
\[
\CO(\delta) = c_{5, 0} \oplus c_{0, 5} \oplus c_{4, 0} \oplus c_{0, 2} \oplus c_{1, 0} \oplus c_{0, 1}
\]
Then the first row of $\CO(\delta)$ is given by picking one sign from each entry, and the remaining rows of $\CO(\delta)$ are generated by the same algorithm for the HC-descent
\[
    \delta' = ++++----+++-.
\]
The first rows of $\CO(\delta)$ and $\CO(\varphi, \chi)$ are the same, and the HC-descent $\delta'$ and the L-descent $\chi' = -+-++-+--+--$ (as computed in Subsection~\ref{ssec:AW-U}) are related by the same formula as in Section \ref{ssec:AV}.

\subsubsection{Special orthogonal groups}
Let $\pi = \pi(\varphi, \chi)$ be a discrete series representation for $\SO(16, 15)$ with character
\[
-++-+-++-+--+++,	
\]
and the auxiliary sign $\delta_0 = -1$.
Then by Section \ref{ssec:AV}, the corresponding character $\delta$ of Harish-Chandra parameter is
\[
--+++++----++-+.	
\]
We can rewrite $\CO(\delta)$ as
\[
\CO(\delta) = c_{1, 0} \oplus c_{0, 1} \oplus c_{2, 0} \oplus c_{0, 4} \oplus c_{5, 0} \oplus c_{0, 2} \oplus c_{0, 1} \oplus c_{0, 2} \oplus c_{5, 0} \oplus c_{0, 4} \oplus c_{2, 0} \oplus c_{0, 1} \oplus c_{1, 0}.
\]
Since $\delta_0 \delta_1 = 1$, the first row of $\CO(\delta)$ is given by picking one sign from each row for index $\geq 2$, and one sign from $c_{0, 2} \oplus c_{0, 1} \oplus c_{0, 2}$, and the remaining rows of $\CO(\delta)$ are generated by the algorithm for even orthogonal group for the HC-descent
\[
    \delta' = --++++---+.
\]
The first rows of $\CO(\delta)$ and $\CO(\varphi, \chi)$ are the same, and the HC-descent $\delta'$ and the L-descent $\chi' = +--+-++-++$ (as computed in Subsection~\ref{ssec:AW-SO}) are related by the same formula as in Section \ref{ssec:AV}.

\section{Proof of the Main Theorem} \label{sec:PF}

Combining the results in previous sections,  we  
prove the following main result of this paper. 

\begin{thm}
\label{Thm:WFS}
Assume that $F=\BR$ and $\pi=\pi(\varphi,\chi)\in \Pi_\BR(G)$ has a generic local $L$-parameter $\varphi$. 
If the component group $\CS_\varphi$ is trivial, Conjecture \ref{conj:wfss} holds. If the component group $\CS_\varphi$ is non-trivial, then 
\[
\wf_\tr(\pi)^{r-\mx}=\wf_\ari(\pi)^{r-\mx}\subset\wf_\wm(\pi)^{s-\mx}.
\]
Moreover, the $\BR$-rational nilpotent orbits in $\wf_\wm(\pi)^{s-\mx}$ lie in a single stable nilpotent orbit. 
\end{thm}

\begin{proof}
Suppose that $\pi=\pi(\varphi,\chi)$ under the local Langlands correspondence  as in Section \ref{sec:ELP}, where $(\varphi,\chi)$ denotes the enhanced $L$-parameter of $\pi$. If the component group $\CS_\varphi$ is trivial, then $G$ is quasi-split and all the three sets 
$\wf_{\wm}(\pi)^{r-\mx}$, $\wf_\ari(\pi)^{r-\mx}$ and $\wf_\tr(\pi)^{r-\mx}$ are equal, consisting of all the regular nilpotent orbits in $\frak g$. Thus the assertion is clear in this case.

Assume that $\CS_\varphi$ is non-trivial. By Corollary \ref{AWF=AV}, 
\[
\wf_\tr(\pi)^{r-\mx}= \wf_\ari(\pi)^{r-\mx}=\{\CO(\varphi,\chi)\}
\]
is a singleton. By Theorem \ref{thm:des=ari}, it holds that $\wf_\ari(\pi)=\wf_\des(\pi)$. Hence from Corollary \ref{cor:wf} it follows that 
\[
\CO(\varphi,\chi)\in\wf_\ari(\pi)=  \wf_\des(\pi) \subset \wf_\wm(\pi).
\]
Recall that by \cite{SV00} we have
\[
{\rm AV}(\pi) = \overline{\CO(\varphi,\chi)_{\frak p}}.
\]
Let $\CV({\rm Ann}\,\pi)\subset \frak g_\BC \cong \Fg_\BC^*$ be the associated variety of the annihilator ideal ${\rm Ann}\,\pi$ of  $\pi$
(which is a two-sided ideal of $U(\frak g_\BC)$). By \cite[Corollary 4]{Mat87}, if $\CO\in \wf_\wm(\pi)$ then 
\[
\CO\subset \CV({\rm Ann}\,\pi).
\]
By \cite[Corollary 4.7]{V91},  $\CV({\rm Ann}\,\pi)$ is the Zariski closure of a single stable nilpotent 
orbit $\CO^{\rm st}(\pi)$, and by \cite[Theorem 8.4]{V91} we have
\[
\av(\pi) \subset \CV({\rm Ann}\,\pi)\cap \frak p,\qquad \CO(\varphi,\chi)_{\frak p} = \CO^{\rm st}(\pi)\cap \frak p.
\]
In particular, if 
$\CO\in \WF_\wm(\pi)^{s-\mx}$ then $\CO\subset \CO^{\rm st}(\pi)$. 
This implies 
\[
\CO(\varphi,\chi) \in \WF_\wm(\pi)^{s-\mx},
\]
and also the last statement of the theorem. The proof of the theorem is finished.
\end{proof}

From the proof of Theorem \ref{Thm:WFS}, we obtain the following consequence. For any $\CO\in\CN_F(\Fg)_\circ$, let $(\ul{p},\{V_{p_i}\})$ be the sesquilinear Young tableau associated with $\CO$ via the correspondence given in \cite{GZ} as in Section \ref{ssec:RNO}. Set $p(\CO):=\ul{p}$, the partition in $(\ul{p},\{V_{p_i}\})$ of $\CO$. 
If $\Omega$ is a subset of $\CN_F(\Fg)_\circ$, we define $p(\Omega)$ accordingly as a set of partitions. 

\begin{cor}\label{cor:WFS-p}
    Assume that $F=\BR$ and $\pi=\pi(\varphi,\chi)\in \Pi_\BR(G)$ has a generic local $L$-parameter $\varphi$.
    Then the following equalities hold:
    \[
    p(\WF_\wm(\pi)^{s-\mx})=p(\WF_\ari(\pi)^{s-\mx})=p(\WF_\tr(\pi)^{s-\mx}),
    \]
    which is a singleton.
\end{cor}

It is clear that $p(\WF_\square(\pi)^{s-\mx})=p(\WF_\square(\pi)^{r-\mx})$ for $\WF_\square(\pi) = \WF_\ari(\pi)$ or $\WF_\tr(\pi)$ 
over $\BR$ considered in this paper, which is expected to hold for $\WF_\wm(\pi)$ as well.

\end{document}